\documentclass[11pt]{article} 

\usepackage[utf8]{inputenc} 

\usepackage{geometry} 
\usepackage{graphicx} 

\usepackage{booktabs} 
\usepackage{array} 
\usepackage{paralist} 
\usepackage{verbatim} 
\usepackage{subfig} 
\usepackage{amsmath}
\usepackage{amssymb}
\usepackage{amsthm}
\usepackage{mathtools}
\usepackage{bbm}
\usepackage{inputenc}

\usepackage{fancyhdr} 
\usepackage{sectsty}

\usepackage[nottoc,notlof,notlot]{tocbibind} 
\usepackage[titles,subfigure]{tocloft} 

\numberwithin{equation}{section}

\theoremstyle{plain}
\newtheorem{thm}{Theorem}[section]
\newtheorem{lem}[thm]{Lemma}
\newtheorem{prop}[thm]{Proposition}
\newtheorem{cor}[thm]{Corollary}
\theoremstyle{definition}
\newtheorem{defn}[thm]{Definition}
\newtheorem{exmp}[thm]{Example}

\theoremstyle{remark}
\newtheorem{rem}[thm]{Remark}

\def\Acal{\mathcal{A}}

\def\Dcal{\mathcal{D}}
\def\Ecal{\mathcal{E}}

\def\Hcal{\mathcal{H}}
\def\Lcal{\mathcal{L}}

\def\Scal{\mathcal{S}}

\def\Ebb{\mathbb{E}}
\def\Nbb{\mathbb{N}}
\def\Pbb{\mathbb{P}}
\def\Qbb{\mathbb{Q}}
\def\Rbb{\mathbb{R}}

\def\Wbb{\mathbb{W}}
\def\Zbb{\mathbb{Z}}
\def\1bb{\mathbbm{1}}

\DeclareMathOperator{\Law}{Law}

\DeclareMathOperator{\HS}{HS}
\let\epsilon\varepsilon
\let\phi\varphi

\title{Existence and space-time regularity for stochastic heat equations on p.c.f. fractals}
\author{Ben Hambly\footnote{Mathematical Institute, University of Oxford, Woodstock Road, Oxford, OX2 6GG, UK. Email: hambly@maths.ox.ac.uk.}\enspace  and Weiye Yang\footnote{Mathematical Institute, University of Oxford, Woodstock Road, Oxford, OX2 6GG, UK. Email: weiye.yang@maths.ox.ac.uk. ORCiD: 0000-0003-2104-1218.}}
\date{} 

\begin{document}
\maketitle
\begin{abstract}
We define linear stochastic heat equations (SHE) on p.c.f.s.s. sets equipped with regular harmonic structures. 
We show that if the spectral dimension of the set is less than two, then function-valued ``random-field'' solutions to 
these SPDEs exist and are jointly H\"{o}lder continuous in space and time. We calculate the respective H\"{o}lder 
exponents, which extend the well-known results on the H\"{o}lder exponents of the solution to SHE on the unit interval. 
This shows that the ``curse of dimensionality'' of the SHE on $\Rbb^n$ depends not on the geometric dimension 
of the ambient space but on the analytic properties of the operator through the spectral dimension. To prove 
these results we establish generic continuity theorems for stochastic processes indexed by these p.c.f.s.s. sets that 
are analogous to Kolmogorov's continuity theorem. We also investigate the long-time behaviour of the solutions to 
the fractal SHEs.
\end{abstract}

\section{Introduction}
The stochastic heat equation (or SHE) on $\Rbb^n, n\in \mathbb{N}$ is a stochastic partial differential equation 
which can be expressed formally as
\begin{equation*}
\begin{split}
\frac{\partial u}{\partial t}(t,x) &= Lu(t,x) + \dot{W}(t,x),\\
u(0,\cdot) &= u_0
\end{split}
\end{equation*}
for $(t,x) \in [0,\infty) \times \Rbb^n$, where $L$ is the Laplacian on $\Rbb^n$, $u_0$ is a (sufficiently regular) function on $\Rbb^n$ and $\dot{W}$ is a space-time white noise on $\Rbb \times \Rbb^n$. Written in the differential notation of stochastic calculus this is equivalent to
\begin{equation*}
\begin{split}
du(t) &= L u(t)dt + dW(t),\\
u(0) &= u_0,
\end{split}
\end{equation*}
where $W$ is a cylindrical Wiener process on $\Lcal^2(\Rbb^n)$. A solution to this SPDE is a process $u = (u(t):t \in [0,T])$ taking values in some space containing $\Lcal^2(\Rbb^n)$ that satisfies the above equations in some weak sense; see \cite{daprato2014} for details. The SHE on $\Rbb^n$ is one of the prototypical examples of an SPDE and has been widely studied, see for example \cite{dawson1972}, \cite{funaki1983} and \cite{walsh1986}. It has two notable properties that are relevant to the present paper. The first is its so-called ``curse of dimensionality''. Solutions to the SHE on $\Rbb^n$ are function-valued only in the case $n = 1$; in dimension $n \geq 2$ solutions are forced to take values in a wider space of distributions on $\Rbb^n$, see \cite{walsh1986}. Secondly if $n = 1$ and $u_0 = 0$ then the solution is unique and jointly H\"{o}lder continuous in space and time, see again \cite{walsh1986}. One of the aims of the present paper is to investigate what happens regarding these two properties in the setting of finitely ramified fractals, which behave in many ways like spaces with dimension between one and two.

The family of spaces that we will be considering is the class of connected \textit{post-critically finite self-similar} (or \textit{p.c.f.s.s.}) sets endowed with \textit{regular harmonic structures}. This family includes many well-known fractals such as the Sierpinski gasket and the Vicsek fractal but not the Sierpinski carpet. The unit interval $[0,1]$ also has several formulations in the language of p.c.f.s.s. sets that belong to this family. Analysis on these sets is a relatively young field which started with the construction of a ``Brownian motion'' on the Sierpinski gasket in \cite{goldstein1987}, \cite{kusuoka1987} and \cite{barlow1988}. This broader theory was then developed and provides a concrete framework where reasonably explicit results can be obtained, see \cite{kigami2001} and \cite{barlow1998}. Associated with a regular harmonic structure on a p.c.f.s.s. set $(F, (\psi_i)_{i=1}^M)$ is an operator, called the \textit{Laplacian} on $F$, which is the generator of a ``Brownian motion'' on $F$ by analogy with the Laplacian on $\Rbb^n$ as the generator of Brownian motion in $\Rbb^n$. We will see that there exists a constant $d_s > 0$ associated with the harmonic structure known as the \textit{spectral dimension}, and it will turn out that the assumption that the harmonic structure is regular implies that $d_s \in [1,2)$. The existence of a Laplacian allows us to define certain PDEs and SPDEs on $F$, such as a heat equation and a stochastic heat equation. The former has been studied extensively, see \cite[Chapter 5]{kigami2001} and further references. The latter is the subject of the present paper.

For examples of some previous work in this area, in \cite{foondun2011} it is shown that on certain fractals a stochastic heat equation can be defined which yields a random-field solution, that is, a solution which is a random map $[0,T] \times F \to \Rbb$. We extend this result in the main theorem of Section \ref{ptwise} of the present paper. In \cite{issoglio2015} (see also \cite{hinz2011}) it is shown that solutions to some nonlinear stochastic heat equations on more general metric measure spaces have H\"{o}lder continuous paths when considered as a random map from a ``time'' set to some space of functions. However in that paper the authors do not consider the H\"{o}lder exponents of the solution when considered as a random field, which is what we will do.

The structure of the present paper is as follows: In the following subsection we describe the precise set-up of the problem and the specific SPDE that we will be studying, and state a theorem which is an important corollary of our main result. In Section \ref{existsec} we recall some useful spectral theory for Laplacians on p.c.f.s.s. sets from \cite{kigami2001} and show that (unique) solutions to the SPDE exist as $\Lcal^2(F)$-valued stochastic processes. In Section \ref{Kolmogorov} we prove generic results analogous to Kolmogorov's continuity theorem for families of random variables indexed by $F$ and by $[0,1] \times F$. In Section \ref{ptwise} we show that the resolvent densities associated with the Laplacian are Lipschitz continuous with respect to the resistance metric on $F$. More importantly we also show that evaluations of solutions to the SPDE at points $(t,x) \in [0,\infty) \times F$ can be done in a well-defined way, which is necessary for us to talk about continuity of these solutions. Section \ref{Holder} contains the main results of the paper, which use our continuity theorems to establish space-time H\"{o}lder continuity of solutions to the SPDE and compute the respective H\"{o}lder exponents. Section \ref{invar} serves as a ``coda'' of the paper, where we prove results on the invariant measures and long-time behaviour of the solutions to the SPDE.

\subsection{Description of the problem}\label{description}

Let $M \geq 2$ be an integer. Let $\Scal = (F, (\psi_i)_{i=1}^M)$ be a connected p.c.f.s.s. set (see \cite{kigami2001}) such that $F$ is a compact metric space and the $\psi_i: F \to F$ are injective strict contractions on $F$. Let $I = \{ 1,\ldots,M \}$ and for each $n \geq 0$ let $\Wbb_n = I^n$. Let $\Wbb_* = \bigcup_{n \geq 0} \Wbb_n$ and let $\Wbb = I^\Nbb$. We call the sets $\Wbb_n$, $\Wbb_*$ and $\Wbb$ \textit{word spaces} and we call their elements \textit{words}. Note that $\Wbb_0$ is a singleton containing an element known as the \textit{empty word}. Words $w \in \Wbb_n$ or $w \in \Wbb$ will be written in the form $w = w_1w_2w_3\ldots$ with $w_i \in I$ for each $i$. For a word $w = w_1, \ldots ,w_n \in \Wbb_*$, let $\psi_w = \psi_{w_1} \circ \cdots \circ \psi_{w_n}$ and let $F_w = \psi_w(F)$.

If $\Wbb$ is endowed with the standard product topology then there is a canonical continuous surjection $\pi: \Wbb \to F$ given in \cite[Lemma 5.10]{barlow1998}. Let $P \subset \Wbb$ be the post-critical set of $\Scal$ (see \cite[Definition 1.3.4]{kigami2001}), which is finite by assumption. Then let $F^0 = \pi(P)$, and for each $n \geq 1$ let $F^n = \bigcup_{w \in \Wbb_n} \psi_w(F^0)$. Let $F_* = \bigcup_{n = 0}^\infty F^n$. It is easily shown that $(F^n)_{n\geq 0}$ is an increasing sequence of finite subsets and that $F_*$ is dense in $F$.

Let the pair $(A_0,\textbf{r})$ be a regular irreducible harmonic structure on $\Scal$ such that $\textbf{r} = (r_1,\ldots,r_M) \in \Rbb^M$ for some constants $r_i > 0$, $i \in I$ (harmonic structures are defined in \cite[Section 3.1]{kigami2001}). Here \textit{regular} means that $r_i \in (0,1)$ for all $i$. Let $r_{\min} = \min_{i \in I} r_i$ and $r_{\max} = \max_{i \in I} r_i$. If $n \geq 0$, $w = w_1,\ldots w_n \in \Wbb_*$ then write $r_w := \prod_{i=1}^n r_{w_i}$. Let $d_H > 0$ be the unique number such that
\begin{equation*}
\sum_{i \in I} r_i^{d_H} = 1.
\end{equation*}
Then let $\mu$ be the self-similar Borel probability measure on $F$ such that for any $n \geq 0$, if $w \in \Wbb_n$ then $\mu(F_w) = r_w^{d_H}$. In other words, $\mu$ is the self-similar measure on $F$ in the sense of \cite[Section 1.4]{kigami2001} associated with the weights $r_i^{d_H}$ on $I$. Let $(\Ecal,\Dcal)$ be the regular local Dirichlet form on $\Lcal^2(F,\mu)$ associated with this harmonic structure, as given by \cite[Theorem 3.4.6]{kigami2001}. This Dirichlet form is associated with a resistance metric $R$ on $F$, defined by
\[ R(x,y) = \left(\inf\{ \Ecal(f,f): f(x)=0, f(y)=1, f\in \Dcal\}\right)^{-1}, \]
which generates the original topology on $F$, by \cite[Theorem 3.3.4]{kigami2001}. Additionally, let
\begin{equation*}
\Dcal_0 = \{ f\in \Dcal: f|_{F^0}=0 \}.
\end{equation*}
Then by \cite[Corollary 3.4.7]{kigami2001}, $(\Ecal,\Dcal_0)$ is a regular local Dirichlet form on $\Lcal^2(F \setminus F^0,\mu)$.

By \cite[Chapter 4]{barlow1998}, associated with the Dirichlet form $(\Ecal,\Dcal)$ on $\Lcal^2(F,\mu)$ is a $\mu$-symmetric diffusion $X^N = (X^N_t)_{t \geq 0}$ which itself is associated with a $C_0$-semigroup of contractions $S^N = (S^N_t)_{t \geq 0}$. Let $L_N$ be the generator of this diffusion. Likewise associated with $(\Ecal,\Dcal_0)$ we have a $\mu$-symmetric diffusion $X^D$ with $C_0$-semigroup of contractions $S^D$ and generator $L_D$. The process $X^D$ is similar to $X^N$, except for the fact that it is absorbed at the points $F^0$, whereas $X^N$ is reflected. The letters $N$ and $D$ indicate \textit{Neumann} and \textit{Dirichlet} boundary conditions respectively. As a consequence of theory developed in \cite[Sections 1.3 and 1.4]{fukushima2010}, the operator $-L_N$ is the non-negative self-adjoint operator associated with the form $(\Ecal,\Dcal)$, in the sense that $\Dcal = \Dcal((-L_N)^{\frac{1}{2}})$ and
\begin{equation*}
\Ecal(f,g) = \langle (-L_N)^{\frac{1}{2}}f,(-L_N)^{\frac{1}{2}}g \rangle_\mu
\end{equation*}
for all $f,g \in \Dcal$. An analogous result holds with $-L_D$ and $(\Ecal,\Dcal_0)$. This justifies us calling $L_N$ the \textit{Neumann Laplacian} and $L_D$ the \textit{Dirichlet Laplacian}.
\begin{exmp}\label{interval}
Let $F = [0,1]$ and take \textit{any} $M \geq 2$. For $1 \leq i \leq M$ let $\psi_i: F \to F$ be the affine map such that $\psi_i(0) = \frac{i-1}{M}$, $\psi_i(1) = \frac{i}{M}$. It follows that $F^0 = \{0,1\}$. Let $r_i = M^{-1}$ for all $i \in I$ and let
\begin{equation*}
A_0 = \left( \begin{array}{cc}
-1 & 1\\
1 & -1
\end{array}
\right).
\end{equation*}
Then all the conditions given above are satisfied. We have $\Dcal = H^1[0,1]$ and $\Ecal(f,g) = \int_0^1 f'g'$. The associated generators $L_N$ and $L_D$ are respectively the standard Neumann and Dirichlet Laplacians on $[0,1]$. In particular, the induced resistance metric $R$ is none other than the standard Euclidean metric. This interpretation of the unit interval as a p.c.f.s.s. set that fits into our set-up will be useful to us later on.
\end{exmp}
The object of study in the present paper is the following SPDE on $F$:
\begin{equation}\label{SPDE}
\begin{split}
du(t) &= L_bu(t)dt + (1-L_b)^{-\frac{\alpha}{2}}dW(t),\\
u(0) &= u_0 \in \Lcal^2(F,\mu),
\end{split}
\end{equation}
where $b \in \{ N, D \}$ and $\alpha \in [0,\infty)$ are parameters and $W$ is a cylindrical Wiener process on $\Lcal^2(F,\mu)$. That is, $W$ formally satisfies
\begin{equation*}
\Ebb \left[ \langle f,W(s) \rangle_{\Lcal^2(F,\mu)}\langle W(t),g \rangle_{\Lcal^2(F,\mu)} \right] = (s \wedge t) \langle f,g \rangle_{\Lcal^2(F,\mu)}
\end{equation*}
for all $s,t \in [0,\infty)$ and $f,g \in \Lcal^2(F,\mu)$. Note that $W$ is not an $\Lcal^2(F,\mu)$-valued process; to be precise, it takes values in some separable Hilbert space in which $\Lcal^2(F,\mu)$ can be continuously embedded (see \cite{daprato2014}). The vast majority of results in this paper hold regardless of the value of $b$; whenever this is not the case it will be explicitly stated.

The SPDE \eqref{SPDE} in the case $\alpha = 0$ will be called the \textit{stochastic heat equation} or \textit{SHE} for $(A_0,\textbf{r})$ on $F$. It is well known (see for example \cite{walsh1986}) that the solution to the standard SHE on $[0,1]$ with initial condition $u_0 = 0$ is jointly continuous with H\"{o}lder exponents of essentially $\frac{1}{2}$ in space and essentially $\frac{1}{4}$ in time (the meaning of ``essentially'' is given in Definition \ref{essentially}). The following extension of this result is a simple consequence of our main result Theorems \ref{SPDEreg2} and \ref{thm:ctsver} and was the original motivation for the writing of the present paper:
\begin{thm}
Equip $F$ with the resistance metric $R$. Then for each $b \in \{ N,D \}$, the SHE for $(A_0,\textbf{r})$ on $F$ with
$u_0 = 0$ has a unique solution $u = (u(t,x))_{(t,x) \in [0,\infty) \times F}$ which is jointly continuous, essentially 
$\frac{1}{2}$-H\"{o}lder continuous in space (i.e. in $(F,R)$) and essentially $\frac{1}{2}(1- \frac{d_s}{2})$-H\"{o}lder 
continuous in time, where
\begin{equation*}
d_s = \frac{2d_H}{d_H + 1}
\end{equation*}
is the spectral dimension of $(F,R)$.
\end{thm}
Note that many p.c.f.s.s. sets $F$ can be embedded into Euclidean space in such a way that $R$ is equivalent to the Euclidean metric up to some exponent. Therefore, for such sets, we can also make sense of the above result with 
respect to a spatial Euclidean metric, see Remark \ref{euclidequiv}.
\begin{exmp}\label{examples}
\begin{enumerate}
\item (Interval.) Take $F = [0,1]$ with the Dirichlet form given in Example \ref{interval}. Then $d_s = 1$ and the resistance metric is the Euclidean metric, so using the above theorem we obtain the usual well-known H\"{o}lder exponents for the SHE on $[0,1]$.
\item ($n$-dimensional Sierpinski gasket.) See \cite[Example 3.1.5]{kigami2001} and \cite[Section 3]{hu2006}. The standard harmonic structure on the $n$-dimensional Sierpinski gasket (for $n \geq 2$) fits into our set-up; it is given by $M = n+1$,
\begin{equation*}
A_0 = \left( \begin{array}{ccccc}
-n & 1 & 1 & \cdots & 1\\
1 & -n & 1 & \cdots & 1\\
1 & 1 & -n & \cdots & 1\\
\vdots &\vdots &\vdots &\ddots &\vdots\\
1 & 1 & 1 & \cdots & -n
\end{array}
\right),
\end{equation*}
and $r_i = \frac{n+1}{n+3}$ for all $i \in I$. In fact for $n=1$ we have the binary decomposition of the unit interval and recover the usual case. For
$n=2$ the diffusion $X^N$ is known as Brownian motion on the Sierpinski gasket and is ubiquitous in the field of analysis 
on fractals (\cite{goldstein1987}, \cite{kusuoka1987}, \cite{barlow1988}). 
We can compute $d_H = \frac{\log(n+1)}{\log(n+3) - \log(n+1)}$ and $d_s = \frac{2\log(n+1)}
{\log(n+3)}$. This gives us a family of examples which live naturally in $\Rbb^n$ for any geometric dimension $n$ and 
where the spectral dimension can be made arbitrarily close to 2 by taking $n$ large. 
Using the properties of the resistance metric we can have solutions that have arbitrarily small spatial (with respect to the Euclidean metric) and temporal H\"{o}lder exponents. See Remark \ref{euclidequiv} for further discussion.
\end{enumerate}
\end{exmp}

\section{Existence (and uniqueness)}\label{existsec}
\begin{defn}\label{mild}
Henceforth we let $\Hcal = \Lcal^2(F,\mu)$. Denote the inner product on $\Hcal$ by $\langle \cdot,\cdot \rangle_\mu$. Let $T>0$. Following \cite{daprato2014}, an $\Hcal$-valued predictable process $u = (u(t): t \in [0,T])$ is a \textit{(mild) solution} to \eqref{SPDE} if
\begin{equation*}
u(t) = S^b_tu_0 + \int_0^t S^b_{t-s}(1-L_b)^{-\frac{\alpha}{2}}dW(s)
\end{equation*}
almost surely for every $t \in [0,T]$. We write $u: [0,T] \to \Hcal$, where we suppress the dependence of $u$ on the underlying probability space. If $T = \infty$ we call the solution \textit{global}.
\end{defn}
\begin{rem}
Global solutions to \eqref{SPDE} are unique up to versions by definition.
\end{rem}
Notice that for any $f \in \Hcal$, $u$ is a solution to \eqref{SPDE} with $u_0 = 0$ if and only if $u + S^bf$ is a solution to \eqref{SPDE} with $u_0 = f$. Thus we can safely assume that $u_0 = 0$, and so we are interested in the properties of the stochastic convolution
\begin{equation}\label{mild0}
W^b_\alpha(t) := \int_0^t S^b_{t-s}(1-L_b)^{-\frac{\alpha}{2}}dW(s).
\end{equation}
Observe that if a solution exists for $u_0 = 0$, then it must equal $W^b_\alpha$ up to versions.

The first thing to investigate is the validity of the operator $(1-L_b)^{-\frac{\alpha}{2}}$ in the case $\alpha > 0$. For an operator $\Acal$ on $\Hcal$, we denote the domain of $\Acal$ by $\Dcal(\Acal)$. If $\Acal$ is bounded then let $\Vert \Acal \Vert$ denote its operator norm. The following statements are immediate by standard operator theory (see \cite[Theorem VIII.5]{reed1981} and \cite[Theorem 12.31]{renardy2004}):
\begin{cor}
For $b \in \{N,D\}$ we have that
\begin{enumerate}
\item $S^b_t = \exp(tL_b)$ for $t \geq 0$,
\item $S^b$ can be extended to an analytic semigroup (which we will identify with $S^b$),
\item For $ \alpha \geq 0$, $(1-L_b)^{-\frac{\alpha}{2}}$ is a bounded linear operator.
\end{enumerate}
\end{cor}
\begin{rem}
The operator $(1-L_b)^{-\frac{\alpha}{2}}$ is known as a \textit{Bessel potential}, see \cite{issoglio2015}, \cite{strichartz2003}.
\end{rem}

\subsection{Spectral theory of Laplacians}
Now that we have established the close relationship between $L_b$ and $\Ecal$, we may make use of the spectral theory of these Laplacians developed in \cite[Chapters 4 and 5]{kigami2001}. We summarise the useful definitions and results below:
\begin{defn}
The unique real $d_H > 0$ such that
\begin{equation*}
\sum_{i \in I} r_i^{d_H} = 1
\end{equation*}
is the \textit{Hausdorff dimension} of $(F,R)$, see \cite[Theorem 1.5.7 and Theorem 4.2.1]{kigami2001}. The \textit{spectral dimension} of $(F,R)$ is given by
\begin{equation*}
d_s = \frac{2d_H}{d_H + 1}.
\end{equation*}
See \cite[Theorem 4.1.5 and Theorem 4.2.1]{kigami2001}.
\end{defn}
\begin{rem}
\begin{enumerate}
\item The definition of $d_s$ given in \cite{kigami2001} is far more general, but the definition above is equivalent for our purposes. We immediately see that $d_s \in (0,2)$ a priori. Were the harmonic structure $(A_0,\textbf{r})$ not regular, it would be possible to have $d_s \geq 2$ via its more general definition.
\item It is possible to show that $d_H \geq 1$. Indeed, by \cite[Theorem 1.6.2 and Lemma 3.3.5]{kigami2001} we have that
\begin{equation*}
\max_{x,y \in F^0} R(x,y) \leq \sum_{i \in I} \max_{x,y \in F^0} R(F_i(x),F_i(y)) \leq \left( \sum_{i \in I} r_i \right) \max_{x,y \in F^0} R(x,y),
\end{equation*}
so that $\sum_{i \in I} r_i \geq 1 = \sum_{i \in I} r_i^{d_H}$ and thus $d_H \geq 1$. It follows that $d_s \in [1,2)$.
\end{enumerate}
\end{rem}

\begin{prop}\label{spectra}
For $b \in \{ N,D \}$ the following statements hold: 

There exists a complete orthonormal basis $(\phi^b_k)_{k=1}^\infty$ of $\Hcal$ consisting of eigenfunctions of the operator $-L_b$. The corresponding eigenvalues $(\lambda^b_k)_{k=1}^\infty$ are non-negative and $\lim_{k \to \infty}\lambda^b_k = \infty$. We assume that they are given in ascending order:
\begin{equation*}
0 \leq \lambda^b_1 \leq \lambda^b_2 \leq \cdots.
\end{equation*}
There exist constants $c_1,c_2,c_3 > 0$ such that if $k \geq 2$ then 
\begin{equation*}
c_1k^\frac{2}{d_s} \leq \lambda^b_k \leq c_2k^\frac{2}{d_s}
\end{equation*}
and
\begin{equation*}
\Vert \phi^b_k \Vert_\infty \leq c_3 |\lambda^b_k|^\frac{d_s}{4}.
\end{equation*}
\end{prop}
\begin{proof}
This is a simple corollary of results in \cite[Chapters 4, 5]{kigami2001}, in particular Theorem 4.5.4 and Lemma 5.1.3.
\end{proof}
\begin{rem}
Note that all functions $f \in \Dcal$ must be at least $\frac{1}{2}$-H\"{o}lder continuous with respect to the resistance metric since
\begin{equation*}
|f(x) - f(y)|^2 \leq \Ecal(f,f)R(x,y)
\end{equation*}
for all $x,y \in F$ (see \cite[Proposition 7.18]{barlow1998}). Thus it makes sense to consider $\phi^b_k(x)$ for $x \in F$. The above proposition then implies that $|\phi^b_k(x)| \leq c_3 |\lambda^b_k|^\frac{d_s}{4}$ for all $x \in F$, $k \geq 2$.
\end{rem}
\begin{rem}\label{consteval}
The reason why we require $k \geq 2$ in the above proposition is that we may have $\lambda^b_1 = 0$. In this case it follows that $\Ecal(\phi^b_1,\phi^b_1) = 0$. By the properties of the resistance metric $R$, for any distinct $x_1,x_2 \in F$ we have that
\begin{equation*}
\frac{|\phi^b_1(x_1) - \phi^b_1(x_2)|^2}{R(x_1,x_2)} \leq \Ecal(\phi^b_1,\phi^b_1) = 0.
\end{equation*}
It follows that $\phi^b_1$ is constant. Since $\Vert \phi^b_1 \Vert_\mu = 1$ and $\mu$ is a probability measure we conclude that $\phi^b_1 \equiv 1$. This confirms that if $0$ is an eigenvalue it must necessarily have multiplicity $1$, so we always have $\lambda^b_2 > 0$. It also implies that we have $\lambda^b_1 = 0$ if and only if $b = N$, since the non-zero constant functions are elements of $\Dcal \setminus \Dcal_0$. In the case that $\lambda^b_1 > 0$, we will assume that $c_1,c_2,c_3$ are chosen such that the estimates in the above proposition hold for $k \geq 1$.
\end{rem}
The existence of a complete orthonormal basis of eigenfunctions of $L_b$ allows us to write down series representations of elements of $\Hcal$ and operators defined on subspaces of $\Hcal$ in a way analogous to the Fourier series representations of elements of $\Lcal^2(0,1)$. For example, an element $f \in \Hcal$ has a series representation
\begin{equation*}
f = \sum_{k=1}^\infty f_k \phi^b_k
\end{equation*}
where $f_k = \langle \phi^b_k,f \rangle_\mu$. Then for any map $\Xi: [0,\infty) \to \Rbb$ we have that the operator $\Xi(-L_b)$ has the representation
\begin{equation*}
\Xi(-L_b)f = \sum_{k=1}^\infty f_k \Xi(\lambda^b_k) \phi^b_k,
\end{equation*}
and the domain of $\Xi(-L_b)$ is exactly those $f \in \Hcal$ for which the above expression is in $\Hcal$. In particular
\begin{equation*}
S^b_tf = \sum_{k=1}^\infty f_k e^{-\lambda^b_kt} \phi^b_k
\end{equation*}
for all $f \in \Hcal$.
\subsection{Existence of solution}
Recall the expression \eqref{mild0}. If we can show that $W^b_\alpha(t) \in \Hcal$ almost surely for every $t > 0$, then we have a unique global solution of \eqref{SPDE} for $u_0 = 0$, and thus by the discussion after Definition \ref{mild} we have a unique global solution for \textit{any} initial value $u_0 \in \Hcal$. In fact we can do better than that:
\begin{defn}\label{essentially}
Let $(M_1,d_1)$ and $(M_2,d_2)$ be metric spaces, and let $f:M_1 \to M_2$ be continuous. For $\delta \in (0,1]$ we say that $f$ is \textit{essentially $\delta$-H\"{o}lder continuous} if it is $\gamma$-H\"{o}lder continuous for every $\gamma < \delta$. That is, for every $\gamma \in (0,\delta)$ there exists a constant $\epsilon_\gamma$ such that $d_2(f(x),f(y)) \leq \epsilon_\gamma d_1(x,y)^\gamma$ for all $x,y \in M_1$.
\end{defn}
\begin{thm}[Existence]\label{existence}
For every $\alpha \geq 0$, $b \in \{N,D\}$ and $T \geq 0$ we have that
\begin{equation*}
\Ebb \left[\Vert W^b_\alpha(T) \Vert_\mu^2 \right] < \infty.
\end{equation*}
In particular for any $\alpha \geq 0$, $b \in \{N,D\}$ and any initial condition $u_0 \in \Hcal$ there exists a unique (up to versions) global solution to \eqref{SPDE}. There exists an $\Hcal$-continuous version of this solution. Moreover if $u_0 = 0$ then this version is essentially $\frac{1}{2}\left( 1 \wedge (1 - \frac{d_s}{2} + \alpha) \right)$-H\"{o}lder continuous on compact intervals. 
\end{thm}
\begin{proof}
We refer to the proof of \cite[Theorem 5.13]{hairer2009}. By It\={o}'s isometry for Hilbert spaces we have that
\begin{equation*}
\Ebb \left[\Vert W^b_\alpha(T) \Vert_\mu^2 \right] = \int_0^T \Vert(1-L_b)^{-\frac{\alpha}{2}} S^b_t \Vert_{\HS}^2 dt,
\end{equation*}
where $\Vert \cdot \Vert_{\HS}$ is the Hilbert-Schmidt norm. If there exists $\beta \in (0,\frac{1}{2} + \frac{\alpha}{2})$ such that $\Vert (1-L_b)^{-\beta} \Vert_{\HS} < \infty$, then by the spectral decomposition of $(1-L_b)^{\beta - \frac{\alpha}{2}} S^b_t$ we have that
\begin{equation}\label{existestim}
\begin{split}
\Vert(1-L_b)^{-\frac{\alpha}{2}} S^b_t \Vert_{\HS} &\leq \Vert (1-L_b)^{-\beta} \Vert_{\HS} \Vert (1-L_b)^{\beta -\frac{\alpha}{2}}S^b_t \Vert\\
&\leq C'(1 \vee t^{\frac{\alpha}{2} - \beta})
\end{split}
\end{equation}
for some constant $C' > 0$, and the last expression is square-integrable on the interval $[0,T]$. Therefore finding such a $\beta$ is sufficient for $W^b_\alpha(t)$ to be square-integrable. We see from Proposition \ref{spectra} that
\begin{equation*}
\begin{split}
\Vert (1-L_b)^{-\beta} \Vert_{\HS}^2 &= \sum_{k=1}^\infty \Vert (1-L_b)^{-\beta} \phi^b_k \Vert_\mu^2\\
&=\sum_{k=1}^\infty (1+ \lambda^b_k)^{-2\beta}\\
&\leq 1+ c_1 \sum_{k=1}^\infty k^{-\frac{4\beta}{d_s}}
\end{split}
\end{equation*}
and the final expression is finite for $\beta > \frac{d_s}{4}$. Since we know that $d_s < 2$ we can pick any $\beta \in (\frac{d_s}{4},\frac{1}{2} + \frac{\alpha}{2}) \neq \emptyset$ to show that $\Ebb \left[\Vert W^b_\alpha(t) \Vert_\mu^2 \right] < \infty$.

For the continuity results, it follows from \eqref{existestim} that for any positive $\gamma < \frac{1}{2}(1 \wedge (1 - \frac{d_s}{2} + \alpha))$ we have that
\begin{equation*}
\int_0^T t^{-2\gamma} \Vert (1-L_b)^{-\frac{\alpha}{2}} S^b_t \Vert_{\HS}^2 dt < \infty.
\end{equation*}
The continuity statements then directly follow from \cite[Theorems 5.10 and 5.17]{hairer2009}.
\end{proof}

\section{Some Kolmogorov-type continuity theorems}\label{Kolmogorov}
It is well-known that solutions to the one-dimensional stochastic heat equation are essentially $\frac{1}{4}$-H\"{o}lder continuous in time and essentially $\frac{1}{2}$-H\"{o}lder continuous in space, so we would like to prove analogous results for our SPDE. It will become clear that the natural ``spatial'' metric to use on $F$ is the resistance metric $R$.

The usual method of proving continuity of processes indexed by $\Rbb$ is to use Kolmogorov's continuity theorem. Our aim in this section is to prove versions of this theorem for the spaces $F$ and $[0,1] \times F$.

\subsection{Partitions and neighbourhoods}
We introduce some more theory and notation from \cite{kigami2001} and develop it further for our purposes.
\begin{defn}
If $n \geq 1$ and $w = w_1 \ldots w_n \in \Wbb_n$ then let
\begin{equation*}
\Sigma_w := \{ w' = w_1' w_2' \ldots \in \Wbb: w'_i = w_i \ \forall i \in \{ 1,\ldots,n\} \}.
\end{equation*}
If $n = 0$ and $w \in \Wbb_0$ then $w$ is the empty word and we set $\Sigma_w := \Wbb$.
\end{defn}
\begin{defn}
A finite subset $\Lambda \subseteq \Wbb_*$ is a \textit{partition} if $\Sigma_w \cap \Sigma_v = \emptyset$ for any $w \neq v \in \Lambda$ and $\Wbb = \bigcup_{w \in \Lambda} \Sigma_w$. A partition $\Lambda$ is a \textit{refinement} of a partition $\Lambda'$ if either $\Sigma_w \subseteq \Sigma_v$ or $\Sigma_w \cap \Sigma_v = \emptyset$ for any $(w,v) \in \Lambda \times \Lambda'$.
\end{defn}
\begin{defn}\label{partitions}
For $a \in (0,1)$ let
\begin{equation*}
\Lambda(a) = \{ w: w = w_1 \ldots w_m \in \Wbb_*,\ r_{w_1\ldots w_{m-1}} > a \geq r_w \}
\end{equation*}
which is a partition, see \cite[Definition 1.5.6]{kigami2001}. Notice that if $w \in \Lambda(a)$ then 
\begin{equation*}
r_{\min}a < r_w \leq a.
\end{equation*}
For $n \geq 1$ let $\Lambda_n = \Lambda(2^{-n})$. Let $\Lambda_0$ be the singleton containing the empty word; this is also a partition.
\end{defn}
\begin{lem}\label{refine}
If $n_1 \geq n_2 \geq 0$ then $\Lambda_{n_1}$ is a refinement of $\Lambda_{n_2}$.
\end{lem}
\begin{proof}
Let $w \in \Lambda_{n_1}$, $v \in \Lambda_{n_2}$ with $\Sigma_w \cap \Sigma_v \neq \emptyset$. Then we must have either $\Sigma_w \subseteq \Sigma_v$ or $\Sigma_v \subseteq \Sigma_w$ (or both). Suppose it is not the case that $\Sigma_w \subseteq \Sigma_v$. Then there exist $m_2 > m_1 \geq 0$ such that $w \in \Wbb_{m_1}$ and $v \in \Wbb_{m_2}$, and $w_i = v_i$ for all $i \in \{1,\ldots,m_1\}$. In particular $v$ is not the empty word, so $w$ is not the empty word (since $n_1 \geq n_2$), so it follows that $m_2 \geq 2$ and $m_1 \geq 1$. But then $n_1,n_2 \geq 1$ so
\begin{equation*}
2^{-n_1} \geq r_w = r_{v_1\ldots v_{m_1}} \geq r_{v_1\ldots v_{m_2-1}} > 2^{-n_2}
\end{equation*}
which is a contradiction. So $\Sigma_w \subseteq \Sigma_v$.
\end{proof}
The above result in particular implies that if $n_1 \geq n_2 \geq 0$ and $v \in \Lambda_{n_1}$ then there exists a $w \in \Lambda_{n_2}$ such that $F_v \subseteq F_w$.
\begin{defn}
For $n \geq 0$ let $F^n_\Lambda = \bigcup_{w \in \Lambda_n} \psi_w(F^0)$. Obviously $F^n_\Lambda \subseteq F_*$ for all $n$. By Lemma \ref{refine} and \cite[Lemma 1.3.10]{kigami2001}, $(F^n_\Lambda)_{n \geq 0}$ is an increasing sequence of subsets.
\end{defn}
\begin{lem}\label{lambdasum}
$\bigcup_{n \geq 0} F^n_\Lambda = F_*$.
\end{lem}
\begin{proof}
Let $n \geq 0$ and $x \in F^n$. Recall the canonical continuous surjection $\pi:\Wbb \to F$ and the post-critical set $P$. By assumption $x \in F^n = \bigcup_{w \in \Wbb_n} \psi_w(F^0) = \bigcup_{w \in \Wbb_n} \psi_w(\pi(P))$, so there exists $w \in \Wbb_n$ and $v \in \Sigma_w$ such that $v_{n+1} v_{n+2}\ldots \in P$ and $\pi(v) = x$. By the definition of $P$ it follows that for all integer $i \geq 0$ we must have that $v_{n+i+1} v_{n+i+2}\ldots \in P$. Now consider the sequence $w^i := v_1\ldots v_{n+i} \in \Wbb_i$ for $i \geq 0$. It follows that $x \in \psi_{w^i}(F^0)$ for all $i \geq 0$. Also some $w^i$ must be in some $\Lambda_m$ for $m \geq 1$, since $\lim_{i \to \infty} r_{w^i} = 0$.
\end{proof}
\begin{defn}
For $n \geq 0$ and $x,y \in F^n_\Lambda$ let $x \sim_n y$ if there exists $w \in \Lambda_n$ such that $x,y \in F_w$. Then $(F^n_\Lambda,\sim_n)$ can be interpreted as a graph.
\end{defn}
\begin{lem}\label{partincrbd}
Suppose that $n \geq 0$, $w \in \Lambda_n$, $v \in \Lambda_{n+1}$ and $\Sigma_v \cap \Sigma_w \neq \emptyset$. If $w \in \Wbb_{m_1}$ and $v \in \Wbb_{m_2}$ then $ 0 \leq m_2 - m_1 < \frac{\log 2 + \log r_{\min}^{-1}}{\log r_{\max}^{-1}}$. In particular if $n_* := \left\lceil \frac{\log 2 + \log r_{\min}^{-1}}{\log r_{\max}^{-1}} \right\rceil$ then $\psi_v(F^0) \subseteq \psi_w(F^{n_*})$.
\end{lem}
\begin{proof}
By the refinement property (Lemma \ref{refine}) we have that $\Sigma_v \subset \Sigma_w$ and so there exist $m_2 \geq m_1 \geq 0$ such that $w \in \Wbb_{m_1}$ and $v \in \Wbb_{m_2}$, and $v_i = w_i$ for all $1 \leq i \leq m_1$. Then by the comment in Definition \ref{partitions},
\begin{equation*}
r_v  > 2^{-(n+1)}r_{\min} \geq \frac{r_{\min}}{2} r_w = \frac{r_{\min}}{2} r_{v_1 \ldots v_{m_1}},
\end{equation*}
so
\begin{equation*}
\frac{r_{\min}}{2} < r_{\max}^{m_2 - m_1}.
\end{equation*}
Thus $m_2 - m_1 < \frac{\log 2 + \log r_{\min}^{-1}}{\log r_{\max}^{-1}}$.
\end{proof}
\begin{lem}\label{graphdist}
There exists a constant $c_g > 0$ such that if $n \geq 0$ and $w \in \Lambda_n$, then $(F^{n+1}_\Lambda \cap F_w, \sim_{n+1})$ is a connected graph and its graph diameter is at most $c_g$.
\end{lem}
\begin{proof}
For $z \in F^{n+1}_\Lambda \cap F_w$, take $\omega \in \pi^{-1}(z) \cap \Sigma_w$ and let $v \in \Lambda_{n+1}$ be such that $\omega \in \Sigma_v$. Then $\Sigma_v \subseteq \Sigma_w$ by the refinement property. By \cite[Proposition 1.3.5(2)]{kigami2001}, $z \in \psi_v(F^0)$. So then by Lemma \ref{partincrbd} we have $z \in \psi_w(F^{n_*})$ for all $z \in F^{n+1}_\Lambda \cap F_w$. Therefore the graph-length of \textit{any} non-self-intersecting path in the graph $(F^{n+1}_\Lambda \cap F_w, \sim_{n+1})$ cannot be greater than $c_g := |F^{n_*}|$. So if we can verify that $(F^{n+1}_\Lambda \cap F_w, \sim_{n+1})$ is connected, we are done.

Consider by the refinement property (Lemma \ref{refine}) that we must have $F_w = \bigcup_{v \in \Lambda'} F_v$, where
\begin{equation*}
\Lambda' = \{ v \in \Lambda_{n+1}: \Sigma_v \subseteq \Sigma_w \} = \{ wv: v \in \Lambda'' \}
\end{equation*}
for some partition $\Lambda''$. With \cite[Proposition 1.3.5(2)]{kigami2001} in mind, the required connectedness result is thus reduced to showing the following: if a graph structure $\sim$ is defined on $\Lambda''$ such that $v \sim v'$ if and only if $F_v \cap F_{v'} \neq \emptyset$, then the graph $(\Lambda'',\sim)$ is connected. This is proven in exactly the same way as \cite[Theorem 1.6.2, (3)$\Rightarrow$(1)]{kigami2001}.
\end{proof}
\begin{defn}\label{nhoods}
Let $n \geq 0$ and $w \in \Wbb_n$. For $x \in F$ let
\begin{equation*}
D^0_n(x) = \bigcup \{ F_w: w \in \Lambda_n,\ F_w \ni x \}
\end{equation*}
be the \textit{$n$-neighbourhood} of $x$. In addition, let
\begin{equation*}
D^1_n(x) = \bigcup \{ F_w: w \in \Lambda_n,\ F_w \cap D^0_n(x) \neq \emptyset \}.
\end{equation*}
By \cite[Lemma 4.2.3]{kigami2001} it must be the case that the quantities $|\{ w \in \Lambda_n: F_w \ni x \}|$ and $|\{ w \in \Lambda_n:F_w \cap D^0_n(x) \neq \emptyset \}|$ are bounded over all $n \geq 0$ and all $x \in F$. Let
\begin{equation}\label{nhoodconst}
c_4 = \max_{n,x}|\{ w \in \Lambda_n:F_w \cap D^0_n(x) \neq \emptyset \}|. 
\end{equation}
\end{defn}
In particular, observe that $D^0_n(x) \subseteq D^1_n(x)$, and that if $x,y \in F^n_\Lambda$ with $x \sim_n y$ then $y \in D^0_n(x)$.
\begin{defn}
For $x \in F$ and $\epsilon > 0$ let $B(x,\epsilon)$ be the closed ball in $(F,R)$ with centre $x$ and radius $\epsilon$.
\end{defn}
The next result shows that the resistance metric $R$ is topologically well-behaved with respect to the structure of the p.c.f.s.s. set $F$ and the partitions $\Lambda_n$. Compare similar results obtained in \cite[Lemmas 3.2, 3.4]{hambly1999}.
\begin{prop}[Homogeneity of resistance metric]\label{nhoodestim}
There exist constants $c_5,c_6 > 0$ such that
\begin{equation*}
B(x,c_5 2^{-n}) \subseteq D^1_n(x) \subseteq B(x,c_6 2^{-n})
\end{equation*}
for all $n \geq 0$ and all $x \in F$.
\end{prop}
\begin{proof}
For the second inclusion, if $y \in D^1_n(x)$ then there exist $w,v \in \Lambda_n$ such that $x \in F_w$, $y \in F_v$ and $F_w \cap F_v \neq \emptyset$. Then the result is a direct consequence of \cite[Proposition 7.18(b)]{barlow1998} and the definition of $\Lambda_n$.

For the first inclusion, let $\Dcal_h \subseteq \Dcal$ be the set of harmonic functions (see \cite[Proposition 3.2.1]{kigami2001}) $f \in \Dcal$ for which $f(x) \in \{ 0,1\}$ for all $x \in F^0$. A harmonic function is completely characterised by the values it takes on $F^0$ so $|\Dcal_h| = 2^{|F^0|}$. Let
\begin{equation*}
c = \max_{f \in \Dcal_h} \Ecal(f,f) > 0.
\end{equation*}
We now take $g$ to be the harmonic extension to $\Hcal$ of the indicator function $\1bb_{D^0_n(x)}|_{F_\Lambda^n}:F^n_\Lambda \to \Rbb$. Then by self-similarity, if $w \in \Lambda_n$ then the function $g \circ \psi_w$ on $F$ agrees exactly with an element of $\Dcal_h$. Evidently $g(x) = 1$, and if $y \notin D^1_n(x)$ then $g(y) = 0$. Therefore it follows by the definition of the resistance metric and the comment in Definition \ref{partitions} that if $y \notin D^1_n(x)$ then
\begin{equation*}
\begin{split}
R(x,y) &\geq \Ecal(g,g)^{-1}\\
&= \left( \sum_{w \in \Lambda_n} r_w^{-1} \Ecal(g \circ \psi_w,g \circ \psi_w) \right)^{-1}\\
&> r_{\min} \left( c_4 2^n c \right)^{-1},
\end{split}
\end{equation*}
where $c_4$ is defined in \eqref{nhoodconst}, and this completes the proof.
\end{proof}
The next result gives bounds on the growth of the cardinality of the sets $\Lambda_n$ in terms of the Hausdorff dimension $d_H$.
\begin{prop}[Cardinality of $\Lambda_n$]\label{partsize}
For all $n \geq 0$,
\begin{equation*}
2^{d_Hn} \leq |\Lambda_n| < r_{\min}^{-d_H} 2^{d_Hn}.
\end{equation*}
\end{prop}
\begin{proof}
For $n \geq 0$ and $v \in \Lambda_n$, by the definition of the measure $\mu$ we have that
\begin{equation*}
r_{\min}^{d_H}2^{-d_Hn} < \mu(F_v) \leq 2^{-d_Hn}.
\end{equation*}
Then summing over all $v \in \Lambda_n$ gives
\begin{equation*}
r_{\min}^{d_H}2^{-d_Hn}|\Lambda_n| < 1 \leq 2^{-d_Hn}|\Lambda_n|.
\end{equation*}
\end{proof}

\subsection{The continuity theorems}
\begin{thm}[First continuity theorem]\label{cty1}
Let $(E,\Delta)$ be a complete separable metric space. Let $\xi = (\xi_x)_{x \in F }$ be an $E$-valued process indexed by $F$ and let $C,\beta,\gamma>0$ such that
\begin{equation*}
\Ebb\left[ \Delta(\xi_x,\xi_y)^\beta \right] \leq C R(x,y)^{d_H + \gamma}
\end{equation*}
for all $x,y \in F$. Then there exists a version of $\xi$ which is almost surely essentially $\frac{\gamma}{\beta}$-H\"{o}lder continuous with respect to $R$.
\end{thm}
\begin{proof}
The set $F$ is uncountable, but $F_* = \bigcup_{n = 0}^\infty F^n_\Lambda$ is countable and dense in $F$. We may therefore consider the countable set $(\xi_x)_{x \in F_* }$ without issues of measurability. Let $\delta \in (0,\frac{\gamma}{\beta})$ and define the measurable event
\begin{equation*}
\Omega_\delta = \left\{ \xi'_x := \lim_{\substack{y \to x\\y \in F_*}}\xi_y \text{ exists $\forall x \in F$ and $x \mapsto \xi'_x$ is $\delta$-H\"{o}lder w.r.t. $(F,R)$} \right\}.
\end{equation*}
We then define the random variables $\hat{\xi}_x$ for $x \in F$ by
\begin{equation*}
\hat{\xi}_x =
\begin{cases}
\begin{array}{ll}
\xi'_x & \text{if }\xi \in \bigcap\{\Omega_\delta: \delta \in \Qbb \cap (0,\frac{\gamma}{\beta}) \},\\
x_0 & \text{otherwise},
\end{array}
\end{cases}
\end{equation*}
for some arbitrary fixed $x_0 \in E$. Then $\hat{\xi} := (\hat{\xi}_x)_{x \in F}$ is measurable and essentially $\frac{\gamma}{\beta}$-H\"{o}lder continuous. If $\Pbb [\Omega_\delta] = 1$ for all $\delta \in \Qbb \cap (0,\frac{\gamma}{\beta})$ then $\hat{\xi}$ is also a version of $\xi$. This is because $\hat{\xi}_x$ is then the almost-sure limit of $(\xi_y)_{y \in F_*}$ as $y \to x$, so applying Fatou's lemma to the estimate in the statement of this theorem shows that $\hat{\xi}_x = \xi_x$ almost surely. It therefore suffices to show that $\Pbb [\Omega_\delta] = 1$ for all $\delta \in (0,\frac{\gamma}{\beta})$. We define the random variable
\begin{equation*}
H_\delta = \sup_{\substack{x,y \in F_*\\x \neq y}} \frac{\Delta(\xi_x,\xi_y)}{R(x,y)^\delta}
\end{equation*}
to be the H\"{o}lder norm of $\xi$ restricted to $F_*$, and we observe that $\Omega_\delta = \{ H_\delta < \infty \}$ by the completeness of $E$. For $n \geq 0$ we also define the random variables
\begin{equation*}
K_n = \sup_{\substack{x,y \in F^n_\Lambda\\x \sim_n y}} \Delta(\xi_x,\xi_y).
\end{equation*}
By Proposition \ref{partsize},
\begin{equation*}
\left\vert\{ (x,y) \in F^n_\Lambda \times F^n_\Lambda : x \sim_n y \}\right\vert \leq |F^0|^2 \cdot |\Lambda_n| \leq |F^0|^2 r_{\min}^{-d_H} 2^{d_H n}.
\end{equation*}
Then using the Markov inequality and Proposition \ref{nhoodestim},
\begin{equation*}
\begin{split}
\Pbb \left[ K_n > 2^{-n\delta} \right] &= \Pbb \left[ K_n^\beta > 2^{-n\delta\beta} \right]\\
&\leq \frac{1}{2}\sum_{\substack{x,y \in F^n_\Lambda\\x \sim_n y}} \Pbb \left[ \Delta(\xi_x,\xi_y)^\beta > 2^{-n\delta\beta} \right]\\
&\leq \frac{2^{n\delta\beta}}{2}\sum_{\substack{x,y \in F^n_\Lambda\\x \sim_n y}} \Ebb \left[ \Delta(\xi_x,\xi_y)^\beta \right]\\
&\leq \frac{C 2^{n\delta\beta}}{2} \sum_{\substack{x,y \in F^n_\Lambda\\x \sim_n y}} R(x,y)^{d_H + \gamma}\\
&\leq C' 2^{d_H n} 2^{-n(d_H + \gamma - \delta\beta)}\\
&= C' 2^{-n(\gamma - \delta\beta)}
\end{split}
\end{equation*}
for some constant $C'>0$. Now $\delta\beta < \gamma$ so
\begin{equation*}
\sum_{n=0}^\infty \Pbb \left[ K_n > 2^{-n\delta} \right] < \infty,
\end{equation*}
so by the Borel-Cantelli lemma we have that $\limsup_{n \to \infty} (2^{n\delta}K_n) \leq 1$ almost surely. In particular there exists an almost surely finite postive random variable $J$ such that $K_n \leq 2^{-n\delta}J$ for all $n \geq 0$ almost surely.

Now recall the constant $c_g$ from Lemma \ref{graphdist}. Let $x,y \in F_*$ be distinct points, and let $m_0$ be the greatest integer such that $y \in D^1_{m_0}(x)$ (which exists by Proposition \ref{nhoodestim}). Then there exists $w,v \in \Lambda_{m_0}$ such that $x \in F_w$, $y \in F_v$ and there exists some $z \in F_w \cap F_v$. In fact by \cite[Proposition 1.3.5]{kigami2001} and the definition of a partition, we can choose $z$ to be in $\psi_w(F^0) \cap \psi_v(F^0)$ so that in particular $z \in F^{m_0}_\Lambda$. Now if $x \in F^{m_0}_\Lambda$ then it follows by Lemma \ref{graphdist} that
\begin{equation*}
\Delta(\xi_x,\xi_z) \leq c_g K_{m_0 + 1}.
\end{equation*}
Otherwise, there exists $m > m_0$ such that $x \in F^m_\Lambda$ and we construct a finite sequence $(x_i)_{i=0}^{m-m_0}$ such that $x_0 = x$, $x_i \in F^{m-i}_\Lambda \cap D^0_{m-i}(x_{i-1}) \cap F_w$ for $i \geq 1$ and $x_{m-m_0} = z$. This can be done in the following way: assume that we already have $x_{i-1} \in F^{m-(i-1)}_\Lambda \cap F_w$ for some $i \in \{ 1,\ldots,m-m_0 \}$. There exists $w^{i-1} \in \Sigma_w$ such that $\pi(w^{i-1}) = x_{i-1}$. Since $\Lambda_{m-i}$ is a partition, there exists $v^{i-1} \in \Lambda_{m-i}$ such that $w^{i-1} \in \Sigma_{v^{i-1}}$. Therefore $F_{v^{i-1}} \ni x_{i-1}$, so we may pick $x_i$ to be some element of $\psi_{v^{i-1}}(F^0)$. By the refinement property $\Sigma_{v^{i-1}} \subseteq \Sigma_w$ so we have that $x_i \in F^{m-i}_\Lambda \cap D^0_{m-i}(x_{i-1}) \cap F_w$. If $i = m - m_0$ then necessarily $v^{m-m_0-1} = w$ and we can specifically choose $x_{m-m_0} = z$.

We then have by Lemma \ref{graphdist} that
\begin{equation*}
\begin{split}
\Delta(\xi_x,\xi_z) &\leq \sum_{i=1}^{m-m_0}\Delta(\xi_{x_i},\xi_{x_{i-1}})\\
&\leq c_g\sum_{i=1}^{m-m_0} K_{m+1-i}\\
&\leq c_g\sum_{n=m_0 + 1}^\infty K_n.
\end{split}
\end{equation*}
We can make the same estimate for $y$ and $z$. Therefore we conclude that
\begin{equation*}
\Delta(\xi_x,\xi_y) \leq 2c_g \sum_{n=m_0+1}^\infty K_n \leq 2c_g J \sum_{n=m_0+1}^\infty 2^{-n\delta} = \frac{2c_g 2^{-(m_0+1)\delta}}{1-2^{-\delta}} J .
\end{equation*}
Now $m_0$ was chosen such that $y \notin D^1_{m_0+1}(x)$, so we use Proposition \ref{nhoodestim} to conclude that $R(x,y) > c_5 2^{-(m_0 + 1)}$. Thus we find that
\begin{equation*}
\frac{\Delta(\xi_x,\xi_y)}{R(x,y)^\delta} \leq C''J
\end{equation*}
for all $x \neq y$ in $F_*$ almost surely, where $C'' > 0$ is a constant. So $H_\delta$ is almost surely finite. So $\Pbb [\Omega_\delta] = 1$.
\end{proof}
\begin{rem}
Taking $F = [0,1]$ as in Example \ref{interval} and $E$ to be the Hilbert space $\Rbb^n$ we obtain the original Kolmogorov continuity theorem.
\end{rem}
We would like the solution to our SPDE to be a (random) map $[0,\infty) \times F \to \Rbb$, so the previous theorem is not quite enough. We now seek to prove a version of it for stochastic processes indexed by $[0,1] \times F$. Let $G$ be the set $[0,1] \times F$ equipped with the natural supremum metric on $\Rbb \times F$ given by
\begin{equation*}
R_\infty((s,x),(t,y)) = \max\{ |s-t|, R(x,y) \}.
\end{equation*}
\begin{prop}\label{cty2}
Let $(E,\Delta)$ be a complete separable metric space. Let $\xi = (\xi_{tx}:(t,x) \in [0,1] \times F )$ be an $E$-valued process indexed by $[0,1] \times F$ and let $C,\beta,\gamma,\gamma'>0$ be such that
\begin{equation*}
\begin{split}
\Ebb\left[ \Delta(\xi_{tx},\xi_{ty})^\beta \right] &\leq C R(x,y)^{d_H + 1 + \gamma},\\
\Ebb\left[ \Delta(\xi_{sx},\xi_{tx})^\beta \right] &\leq C |s-t|^{d_H + 1 + \gamma'}
\end{split}
\end{equation*}
for all $s,t \in [0,1]$ and all $x,y \in F$. Then there exists a version of $\xi$ which is almost surely essentially $\frac{\gamma \wedge \gamma'}{\beta}$-H\"{o}lder continuous with respect to $G = ([0,1] \times F,R_\infty)$.
\end{prop}
\begin{proof}
This proof proceeds in much the same way as in Theorem \ref{cty1}, so we only give an outline.

For $n \geq 0$ we let
\begin{equation*}
G^n = \{ k 2^{-n}: k= 0,1,\ldots, 2^n \} \times F^n_\Lambda
\end{equation*}
and
\begin{equation*}
G_* = \bigcup_{n=0}^\infty G^n,
\end{equation*}
then $G_*$ is countable and dense in $G$. Then for each $n \geq 0$ we define a relation $\ast_n$ on $G^n$ by $(s,x) \ast_n (t,y)$ if and only if either ($|s-t| = 2^{-n}$ and $x = y$) or ($s = t$ and $x \sim_n y$). Notice that this implies that if $(s,x) \ast_n (t,y)$ then $R_\infty((s,x),(t,y)) \leq (c_6 \vee 1)2^{-n}$, by Proposition \ref{nhoodestim}. Then as before we can define
\begin{equation*}
K_n = \sup_{\substack{p,q \in G^n\\p \ast_n q}} \Delta(\xi_p,\xi_q).
\end{equation*}
Since both $[0,1]$ and $(F,R)$ are bounded, for $\delta \in (0, \frac{\gamma \wedge \gamma'}{\beta})$ this satisfies
\begin{equation*}
\begin{split}
\Pbb \left[ K_n > 2^{-n\delta} \right] &= \Pbb \left[ K_n^\beta > 2^{-n\delta\beta} \right]\\
&\leq \frac{1}{2}\sum_{\substack{p,q \in G^n\\p \ast_n q}} \Pbb \left[ \Delta(\xi_p,\xi_q)^\beta > 2^{-n\delta\beta} \right]\\
&\leq \frac{2^{n\delta\beta}}{2}\sum_{\substack{p,q \in G^n\\p \ast_n q}} \Ebb \left[ \Delta(\xi_p,\xi_q)^\beta \right]\\
&\leq 2^{\beta - 1} 2^{n\delta\beta}\sum_{\substack{(s,x),(t,y) \in G^n\\(s,x) \ast_n (t,y)}} \Ebb \left[ \Delta(\xi_{sx},\xi_{tx})^\beta + \Delta(\xi_{tx},\xi_{ty})^\beta \right]\\
&\leq C 2^{\beta - 1} 2^{n\delta\beta} \sum_{\substack{(s,x),(t,y) \in G^n\\(s,x) \ast_n (t,y)}} \left( |s-t|^{d_H + 1 + \gamma'} + R(x,y)^{d_H + 1 + \gamma} \right)\\
&\leq C' 2^{n\delta\beta} \sum_{\substack{(s,x),(t,y) \in G^n\\(s,x) \ast_n (t,y)}} R_\infty((s,x),(t,y))^{d_H + 1 + \gamma \wedge \gamma'}\\
&\leq C'' 2^{d_H n} 2^n 2^{-n(d_H + 1 + \gamma \wedge \gamma' - \delta\beta)}\\
&= C'' 2^{-n(\gamma \wedge \gamma' - \delta\beta)}.
\end{split}
\end{equation*}
So as in Theorem \ref{cty1}, there exists an almost surely finite positive random variable $J$ such that $K_n \leq 2^{-n\delta}J$ for all $n \geq 0$ almost surely. The sets analogous to $D^0_n(x)$ and $D^1_n(x)$ in Theorem \ref{cty1} are given by
\begin{equation*}
\hat{D}^0_n(s,x) = ([s_-,s^-] \cap [0,1]) \times D^0_n(x)
\end{equation*}
and
\begin{equation*}
\hat{D}^1_n(s,x) = ([s_- - 2^{-n}, s^- + 2^{-n}] \cap [0,1]) \times D^1_n(x)
\end{equation*}
where $s_- = \max \{ k2^{-n}: k \in \Zbb,\ k2^{-n}< s \}$, $s^- = \min \{ k2^{-n}: k \in \Zbb,\ k2^{-n}> s \}$. Using Proposition \ref{nhoodestim} it is simple to verify the analogous result that
\begin{equation*}
B_\infty(p,(c_5 \wedge 1) 2^{-n}) \subseteq \hat{D}^1_n(p) \subseteq B_\infty(p,(c_6 \vee 2) 2^{-n})
\end{equation*}
for all $n \geq 0$ and all $p \in G$, where $B_\infty$ denotes the closed $R_\infty$-balls of $G$. Now if $(s,x), (t,y) \in G_*$ are distinct points, let $m_0$ be the greatest integer such that $(t,y) \in \hat{D}^1_{m_0}(s,x)$. Then there exists $w,v \in \Lambda_{m_0}$ and $\tau_1,\tau_2 \in \{ k2^{-m_0}: k= 0,1,\ldots,2^{m_0} - 1 \}$ such that
\begin{equation*}
\begin{split}
(s,x) &\in [\tau_1,\tau_1 + 2^{-m_0}] \times F_w,\\
(t,y) &\in [\tau_2,\tau_2 + 2^{-m_0}] \times F_v,
\end{split}
\end{equation*}
and there exists some
\begin{equation*}
(\tau,z) \in [\tau_1,\tau_1 + 2^{-m_0}] \times F_w \cap [\tau_2,\tau_2 + 2^{-m_0}] \times F_v.
\end{equation*}
In fact just as in the proof of Theorem \ref{cty1} we may pick $(\tau,z)$ such that
\begin{equation*}
(\tau,z) \in \{\tau_1,\tau_1 + 2^{-m_0}\} \times \psi_w(F^0) \cap \{\tau_2,\tau_2 + 2^{-m_0}\} \times \psi_v(F^0) \subset G^{m_0}.
\end{equation*}
We can then estimate $\Delta(\xi_{sx},\xi_{ty})$ by constructing suitable finite sequences of points from $(s,x)$ to $(\tau,z)$ and from $(t,y)$ to $(\tau,z)$, similar to the proof of Theorem \ref{cty1}.

The rest of the details of the proof are left up to the reader; it suffices to adapt the proof of Theorem \ref{cty1}, using for example $c_g + 2$ instead of $c_g$. 
\end{proof}
We now extend the previous result to this section's main theorem, which includes spatial and temporal H\"{o}lder exponents.
\begin{thm}[Second continuity theorem]\label{cty3}
Let $(E,\Delta)$ be a complete separable metric space. Let $\xi = (\xi_{tx}:(t,x) \in [0,1] \times F )$ be an $E$-valued process indexed by $[0,1] \times F$ and let $C,\beta,\gamma,\gamma'>0$ be such that
\begin{equation}\label{ctyestim}
\begin{split}
\Ebb\left[ \Delta(\xi_{tx},\xi_{ty})^\beta \right] &\leq C R(x,y)^{d_H + 1 + \gamma},\\
\Ebb\left[ \Delta(\xi_{sx},\xi_{tx})^\beta \right] &\leq C |s-t|^{d_H + 1 + \gamma'}
\end{split}
\end{equation}
for all $s,t \in [0,1]$ and all $x,y \in F$. Then there exists a version $\hat{\xi} = (\hat{\xi}_{tx}:(t,x) \in [0,1] \times F )$ of $\xi$ which satisfies the following:
\begin{enumerate}
\item The map $(t,x) \mapsto \hat{\xi}_{tx}$ is almost surely essentially $\delta_0$-H\"{o}lder continuous with respect to $R_\infty$ where
\begin{equation*}
\delta_0 = \frac{1}{\beta} \left( \gamma \wedge \gamma' \right).
\end{equation*}
\item For every $t \in [0,1]$ the map $x \mapsto \hat{\xi}_{tx}$ is almost surely essentially $\delta_1$-H\"{o}lder continuous with respect to $R$ where
\begin{equation*}
\delta_1 = \frac{1}{\beta} \left( 1 + \gamma \right).
\end{equation*}
\item For every $x \in F$ the map $t \mapsto \hat{\xi}_{tx}$ is almost surely essentially $\delta_2$-H\"{o}lder continuous with respect to the Euclidean metric where
\begin{equation*}
\delta_2 = \frac{1}{\beta} \left( d_H + \gamma' \right).
\end{equation*}
\end{enumerate}
\end{thm}
\begin{proof}
(1) is exactly Proposition \ref{cty2}. For (2) we fix $t \in [0,1]$. We see that the space increment estimate is equivalent to
\begin{equation*}
\Ebb\left[ \Delta(\hat{\xi}_{tx},\hat{\xi}_{ty})^\beta \right] \leq C R(x,y)^{d_H + \beta\delta_1}.
\end{equation*}
Then by Theorem \ref{cty1} there exists a version $(\tilde{\xi}_x)_{x \in F}$ of $(\hat{\xi}_{tx})_{x \in F}$ which is almost surely essentially $\delta_1$-H\"{o}lder continuous with respect to $R$. Now using (1), $(\tilde{\xi}_x)_{x \in F}$ and $(\hat{\xi}_{tx})_{x \in F}$ are both almost surely continuous on the separable space $F$ (see for example $F_* \subseteq F$) so we must in fact have that $(\tilde{\xi}_x)_{x \in F} = (\hat{\xi}_{tx})_{x \in F}$ almost surely. We conclude that $(\hat{\xi}_{tx})_{x \in F}$ is almost surely essentially $\delta_1$-H\"{o}lder continuous with respect to $R$. The proof of (3) is conceptually identical --- we use the standard Kolmogorov continuity theorem for $[0,1]$.
\end{proof}
\begin{rem}
This time by taking $F = [0,1]$ as in Example \ref{interval} the above theorem reduces to the original Kolmogorov continuity theorem for $[0,1]^2$.
\end{rem}
\begin{cor}\label{cty4}
Theorem \ref{cty3} holds if the interval $[0,1]$ is replaced with $[0,T]$ for any $T > 0$.
\end{cor}
\begin{proof}
We have that $\xi = (\xi_{tx}:(t,x) \in [0,T] \times F )$ is an $E$-valued process indexed by $[0,T] \times F$. By taking a linear rescaling $[0,T] \leftrightarrow [0,1]$ of the time coordinate we transform $\xi$ into an $E$-valued process indexed by $[0,1] \times F$ with the same exponents in the continuity estimates \eqref{ctyestim}. Then we use Theorem \ref{cty3} to construct a H\"{o}lder continuous version of the rescaled $\xi$. Finally we reverse the rescaling, which is linear so it preserves H\"{o}lder exponents.
\end{proof}

\section{Pointwise regularity}\label{ptwise}

Before we talk about H\"{o}lder continuity of the solution $u$ to \eqref{SPDE} we show that the point evaluations $u(t,x)$ for $(t,x) \in [0,\infty) \times F$ are indeed well-defined random variables. Recall from Proposition \ref{spectra} and the subsequent discussion that
\begin{equation*}
S^b_t(1-L_b)^{-\frac{\alpha}{2}}f = \sum_{k=1}^\infty f_k e^{-\lambda^b_kt}(1+ \lambda^b_k)^{-\frac{\alpha}{2}} \phi^b_k
\end{equation*}
for all $f \in \Hcal = \Lcal^2(F,\mu)$, where $f_k = \langle \phi^b_k,f \rangle_\mu$. Equivalently
\begin{equation*}
S^b_t(1-L_b)^{-\frac{\alpha}{2}} = \sum_{k=1}^\infty e^{-\lambda^b_kt}(1+ \lambda^b_k)^{-\frac{\alpha}{2}} \phi^b_k \phi^{b*}_k
\end{equation*}
where $\phi^{b*}_k \in \Hcal^*$ is the bounded linear functional $f \mapsto \langle \phi^b_k,f \rangle_\mu$. By Proposition \ref{spectra} we have that
\begin{equation*}
\begin{split}
\sum_{k=1}^\infty (1+ \lambda^b_k)^{-\alpha} \int_0^t \Vert e^{-\lambda^b_k(t-s)} \phi^{b*}_k \Vert_{\HS}^2 ds = \sum_{k=1}^\infty \frac{1 - e^{-2\lambda^b_kt}}{2\lambda^b_k(1+ \lambda^b_k)^{\alpha}} \leq C \sum_{k=1}^\infty k^{-\frac{2}{d_s}} < \infty,
\end{split}
\end{equation*}
so it follows from It\={o}'s isometry for $\Hcal$-valued stochastic integrals that
\begin{equation*}
\begin{split}
W^b_\alpha(t) &:= \int_0^t S^b_{t-s}(1-L_b)^{-\frac{\alpha}{2}}dW(s)\\
&= \sum_{k=1}^\infty \int_0^t e^{-\lambda^b_k(t-s)} \phi^{b*}_k dW(s) (1+ \lambda^b_k)^{-\frac{\alpha}{2}} \phi^b_k.
\end{split}
\end{equation*}
For each $k \geq 1$ define the real-valued stochastic process $X^{b,k} = (X^{b,k}_t)_{t \geq 0}$ by
\begin{equation}\label{OUproc}
X^{b,k}_t = \int_0^t e^{-\lambda^b_k(t-s)} \phi^{b*}_k dW(s)
\end{equation}
so we have the series representation
\begin{equation}\label{Wseries}
W^b_\alpha(t) = \sum_{k=1}^\infty (1+ \lambda^b_k)^{-\frac{\alpha}{2}} X^{b,k}_t \phi^b_k.
\end{equation}
Evidently $X^{b,k}$ is a centred real continuous Gaussian process. We compute its covariance to be
\begin{equation*}
\Ebb \left[ X^{b,k}_tX^{b,k}_{t+s} \right] = \frac{e^{-\lambda^b_ks}}{2\lambda^b_k}(1 - e^{-2\lambda^b_kt})
\end{equation*}
if $\lambda^b_k > 0$ and we identify $X^{b,k}$ to be a centred Ornstein-Uhlenbeck process with unit volatility and rate parameter $\lambda^b_k$. If $\lambda^b_k = 0$ then $X^{b,k}$ is simply a standard Wiener process. It is easy to check that the family $(X^{b,k})_{k=1}^\infty$ is independent.
\begin{rem}
We give an alternative view on the series representation \eqref{Wseries}. Let $u$ be the solution to \eqref{SPDE} in the case $u_0 = 0$, so that $u = W^b_\alpha$. We take an eigenfunction expansion of \eqref{SPDE}:
\begin{equation}\label{SPDEexp}
\begin{split}
d\hat{u}(t,k) &= -\lambda^b_k \hat{u}(t,k)dt + (1+ \lambda^b_k)^{-\frac{\alpha}{2}}\phi^{b*}_k dW(t),\\
\hat{u}(0,k) &= 0
\end{split}
\end{equation}
for each $k \in \Nbb$, where $\hat{u}(\cdot,k) := \langle \phi^b_k, u(\cdot) \rangle_\mu$ is a real-valued process. This is analogous to using Fourier methods to solve differential equations on $\Rbb^n$. Now using standard theory we see that $\{ \phi^{b*}_k W \}_{k=1}^\infty$ is a family of independent real-valued standard Wiener processes. It follows that \eqref{SPDEexp} is just a family of decoupled one-dimensional SDEs, and the solution to the $k$th SDE can be found to be exactly $(1+ \lambda^b_k)^{-\frac{\alpha}{2}} X^{b,k}$.
\end{rem}

\subsection{Resolvent density}
\begin{defn}
If $\lambda > 0$ then $\Dcal$ can be equipped with the inner product
\begin{equation*}
\langle f,g \rangle_\lambda := \Ecal(f,g) + \lambda \langle f,g \rangle_\mu.
\end{equation*}
Since $\Ecal$ is a closed form, this turns $\Dcal$ into a Hilbert space which we denote $\Dcal^\lambda$. Observe that the evaluation maps $\{f \mapsto f(x):\ x \in F\}$ are continuous linear functionals on $\Dcal^\lambda$, by \cite[Proposition 7.16(b)]{barlow1998}. We have that $\Dcal_0$ is the intersection of the kernels of $\{f \mapsto f(x):\ x \in F^0 \}$ so it must be closed with respect to $\langle \cdot , \cdot \rangle_\lambda$.
\end{defn}
\begin{defn}
For $\lambda > 0$ and $b \in \{ N,D \}$ let $\rho^b_\lambda: F \times F \to \Rbb$ be the \textit{resolvent density} associated with $L_b$. By \cite[Theorem 7.20]{barlow1998}, $\rho^N_\lambda$ exists and satisfies the following:
\begin{enumerate}
\item (Reproducing kernel property.) For $x \in F$, $\rho^N_\lambda(x,\cdot)$ is the unique element of $\Dcal$ such that
\begin{equation*}
\langle \rho^N_\lambda(x,\cdot), f \rangle_\lambda = f(x)
\end{equation*}
for all $f \in \Dcal$.
\item (Resolvent kernel property.) For all continuous $f \in \Hcal$ and all $x \in F$,
\begin{equation*}
\int_0^\infty e^{-\lambda t} S^N_tf(x) dt = \int_F \rho^N_\lambda(x,y)f(y)\mu(dy).
\end{equation*}
By a density argument it follows that for all $f \in \Hcal$,
\begin{equation*}
\int_0^\infty e^{-\lambda t} S^N_tf dt = \int_F \rho^N_\lambda(\cdot,y)f(y)\mu(dy).
\end{equation*}
\item $\rho^N_\lambda$ is non-negative (easy to see from (2) and fact that $S^N_tf(x) = \Ebb^x[f(X^N_t)]$), symmetric and bounded. We define (for now) $c_7(\lambda) > 0$ such that
\begin{equation*}
c_7(\lambda) \geq \sup_{x,y \in F}\rho^N_\lambda(x,y).
\end{equation*}
\item (H\"{o}lder continuity.) For this same constant $c_7(\lambda)$ we have that for all $x,y,y' \in F$,
\begin{equation*}
|\rho^N_\lambda(x,y) - \rho^N_\lambda(x,y')|^2 \leq c_7(\lambda) R(y,y').
\end{equation*}
Using symmetry this H\"{o}lder continuity result holds in the first argument as well.
\end{enumerate}
By an identical argument to \cite[Theorem 7.20]{barlow1998}, $\rho^D_\lambda$ exists and satisfies the analogous results with $(\Ecal,\Dcal_0)$ and $S^D$. By the reproducing kernel property it follows that for every $x \in F$, $\rho^D_\lambda(x,\cdot)$ must be the $\Dcal^\lambda$-orthogonal projection of $\rho^N_\lambda(x,\cdot)$ onto $\Dcal_0$. We now choose $c_7(\lambda)$ large enough that it does not depend on the value of $b \in \{ N,D \}$ for (3) and (4).
\end{defn}
The H\"{o}lder continuity property of the resolvent densities $\rho^b_\lambda$ described above is the subject of this section. We seek to strengthen it into Lipschitz continuity. 
\begin{defn}
Let $B \subseteq F$ be closed. Let $g_B: F \times F \to \Rbb$ be the \textit{Green function} on $F$ with boundary $B$, see \cite[Chapter 4]{kigami2012}.
\end{defn}
The properties of the Green function that we require are given in \cite[Theorem 4.1]{kigami2012}. Note in particular that every Green function is symmetric and uniformly Lipschitz in $(F,R)$; this is the main tool of our proof.
\begin{prop}[Lipschitz resolvent]\label{resolvLip}
For $\lambda > 0$ and $b \in \{ N,D \}$, if $x,y,y' \in F$ then
\begin{equation*}
\left\vert \rho^b_\lambda(x,y) - \rho^b_\lambda(x,y') \right\vert \leq 2 R(y,y').
\end{equation*}
\end{prop}
\begin{proof}
Let $\textbf{1} \in \Hcal$ be the constant function taking the value $1$. First of all, observe that for all $\lambda > 0$, $x \in F$ and $b \in \{ N,D \}$,
\begin{equation*}
\langle \rho^b_\lambda(x,\cdot) , \textbf{1} \rangle_\mu = \int_0^\infty e^{-\lambda t} S^b_t \textbf{1}(x) dt = \int_0^\infty e^{-\lambda t} \Pbb^x [X^b_t \in F] dt \leq \frac{1}{\lambda}.
\end{equation*}

Fix $\lambda > 0$. We prove the result for $b = D$ first. Let $g_D = g_{F^0}$, the Green function associated with Dirichlet boundary conditions. Then for $x,y \in F$,
\begin{equation*}
\rho^D_\lambda(x,y) = \Ecal\left( \rho^D_\lambda(x,\cdot), g_D(y,\cdot) \right) = g_D(y,x) - \lambda \langle \rho^D_\lambda(x,\cdot), g_D(y,\cdot) \rangle_\mu.
\end{equation*}
So if $x,y,y' \in F$ then by \cite[Theorem 4.1]{kigami2012} and the non-negativity of the resolvent density,
\begin{equation*}
\begin{split}
\rho^D_\lambda(x,y) - \rho^D_\lambda(x,y') &\leq |g_D(y,x) - g_D(y',x)| + \lambda \int_F \rho^D_\lambda(x,z) |g_D(y,z) - g_D(y',z)| \mu(dz)\\
&\leq R(y,y')\left( 1 + \lambda \int_F \rho^D_\lambda(x,z) \textbf{1}(z) \mu(dz) \right)\\
&\leq 2R(y,y').
\end{split}
\end{equation*}
Doing the same estimate with $y,y'$ interchanged gives the required result. Now for the case $b = N$, fix $x_0 \in F$ an arbitrary point. We see that
\begin{equation*}
\rho^N_\lambda(x,y) - \rho^N_\lambda(x,x_0) = \Ecal\left( \rho^N_\lambda(x,\cdot), g_{\{x_0\}}(y,\cdot) \right) = g_{\{x_0\}}(y,x) - \lambda \langle \rho^N_\lambda(x,\cdot), g_{\{x_0\}}(y,\cdot) \rangle_\mu,
\end{equation*}
and the rest of the proof is identical to the $b = D$ case.
\end{proof}
As before, by the symmetry of $\rho^b_\lambda$ the above Lipschitz continuity property in fact holds in both of its arguments.

\subsection{Pointwise regularity of solution}
We return to the SPDE \eqref{SPDE}. The next lemma is based on an argument in \cite[Section 7.2]{foondun2011}.
\begin{lem}\label{resolvestim}
Let $u: [0,\infty) \to \Hcal$ be the solution to \eqref{SPDE} with initial condition $u_0 = 0$. If $g \in \Hcal$ and $t \in [0,\infty)$ then
\begin{equation*}
\Ebb \left[ \langle u(t),g \rangle_\mu^2 \right] \leq \frac{e^{2t}}{2} \int_F \int_F \rho^b_1(x,y)g(x)g(y)\mu(dx)\mu(dy).
\end{equation*}
\end{lem}
\begin{proof}
Let $g^* \in \Hcal^*$ be the bounded linear functional $f \mapsto \langle f,g \rangle_\mu$. We see by It\={o}'s isometry that
\begin{equation*}
\begin{split}
\Ebb \left[ \langle u(t),g \rangle_\mu^2 \right] &= \Ebb \left[ g^*(u(t))^2 \right]\\
&= \int_0^t \Vert g^* (1-L_b)^{-\frac{\alpha}{2}}S^b_s \Vert_{\HS}^2 ds\\
&= \int_0^t \Vert (1-L_b)^{-\frac{\alpha}{2}}S^b_s g \Vert_\mu^2 ds\\
\end{split}
\end{equation*}
where the last equality is a result of the self-adjointness of the operator $(1-L_b)^{-\frac{\alpha}{2}}S^b_s$. We know from the functional calculus for self-adjoint operators that $\Vert (1-L_b)^{-\frac{\alpha}{2}} \Vert \leq 1$ so
\begin{equation*}
\begin{split}
\Ebb \left[ \langle u(t),g \rangle_\mu^2 \right] &\leq \int_0^t \Vert S^b_s g \Vert_\mu^2 ds\\
&\leq e^{2t} \int_0^\infty e^{-2s} \Vert S^b_s g \Vert_\mu^2 ds\\
&= e^{2t} \left\langle \int_0^\infty e^{-2s} S^b_{2s}g ds , g \right\rangle_\mu\\
&= \frac{e^{2t}}{2} \left\langle \int_F \rho^b_1(\cdot,y) g(y) \mu(dy) , g \right\rangle_\mu\\
&= \frac{e^{2t}}{2} \int_F \int_F \rho^b_1(x,y)g(x)g(y)\mu(dx)\mu(dy).
\end{split}
\end{equation*}
\end{proof}
\begin{defn}
For $x \in F$ and $n \geq 0$, define
\begin{equation*}
f^x_n = \mu(D^0_n(x))^{-1} \1bb_{D^0_n(x)},
\end{equation*}
see \cite[Section 7.2]{foondun2011}.
\end{defn}
Evidently $f^x_n \in \Hcal$, $\Vert f^x_n \Vert_\mu^2 = \mu(D^0_n(x))^{-1} < r_{\min}^{-d_H} 2^{d_Hn}$ (by the definition of $d_H$ and the comment in Definition \ref{partitions}) and if $g \in \Hcal$ is continuous then
\begin{equation*}
\lim_{n \to \infty}\langle f^x_n,g \rangle_\mu = g(x),
\end{equation*}
by Proposition \ref{nhoodestim}. We can now state and prove the main theorem of this section.
\begin{thm}[Pointwise regularity]\label{ptreg}
Let $u: [0,\infty) \to \Hcal$ be the solution to the SPDE \eqref{SPDE} with initial value $u_0 = 0$. Then for all $(t,x) \in [0,\infty) \times F$ the expression
\begin{equation*}
u(t,x) := \sum_{k=1}^\infty (1+ \lambda^b_k)^{-\frac{\alpha}{2}} X^{b,k}_t \phi^b_k(x)
\end{equation*}
is a well-defined real-valued centred Gaussian random variable. There exists a constant $c_8 > 0$ such that for all $x \in F$, $t \in [0,\infty)$ and $n \geq 0$ we have that
\begin{equation*}
\Ebb \left[ \left( \langle u(t), f^x_n \rangle_\mu - u(t,x) \right)^2 \right] \leq c_8e^{2t} 2^{-n}.
\end{equation*}
\end{thm}
\begin{proof}
Note that $\phi^b_k \in \Dcal(L_b)$ for each $k$, so $\phi^b_k$ is continuous and so $\phi^b_k(x)$ is well-defined. By the definition of $u(t,x)$ as a sum of real-valued centred Gaussian random variables we need only prove that it is square-integrable and that the approximation estimate holds. Let $x \in F$. The theorem is trivial for $t = 0$ so let $t \in (0,\infty)$. By Lemma \ref{resolvestim} we have that
\begin{equation*}
\begin{split}
\Ebb &\left[ \langle u(t),f^x_n - f^x_m \rangle_\mu^2 \right]\\
&\leq \frac{e^{2t}}{2} \int_F \int_F \rho^b_1(z_1,z_2)(f^x_n(z_1)-f^x_m(z_1))(f^x_n(z_2)-f^x_m(z_2))\mu(dz_1)\mu(dz_2).
\end{split}
\end{equation*}
Then using the definition of $f^x_n$, Proposition \ref{resolvLip} and Proposition \ref{nhoodestim} we have that
\begin{equation}\label{cauchyseq}
\begin{split}
\Ebb \left[ \langle u(t),f^x_n - f^x_m \rangle_\mu^2 \right] &\leq \frac{e^{2t}}{2} \left( 4c_62^{-n} + 4c_62^{-m} \right)\\
&= 2e^{2t}c_6 \left( 2^{-n} + 2^{-m} \right).
\end{split}
\end{equation}
Writing $u$ in its series representation \eqref{Wseries} and using the independence of the $X^{b,k}$ and the fact that $\sum_{k=1}^\infty \Ebb \left[(X^{b,k}_t)^2 \right] < \infty$, this is equivalent to
\begin{equation*}
\sum_{k=1}^\infty (1+ \lambda^b_k)^{-\alpha} \Ebb \left[(X^{b,k}_t)^2 \right] \left( \langle \phi^b_k, f^x_n\rangle_\mu - \langle \phi^b_k,f^x_m \rangle_\mu \right)^2 \leq 2e^{2t}c_6 \left( 2^{-n} + 2^{-m}\right).
\end{equation*}
It follows that the left-hand side tends to zero as $m,n \to \infty$. By Theorem \ref{existence} we know that
\begin{equation*}
\sum_{k=1}^\infty (1+ \lambda^b_k)^{-\alpha} \Ebb \left[(X^{b,k}_t)^2 \right] \langle \phi^b_k, f^x_n\rangle_\mu^2 = \Ebb \left[ \langle u(t),f^x_n \rangle_\mu^2 \right] < \infty
\end{equation*}
for all $x \in F$, $n \geq 0$ and $t \in [0,\infty)$, therefore by the completeness of the sequence space $\ell^2$ there must exist a unique sequence $(y_k)_{k=1}^\infty$ such that $\sum_{k=1}^\infty y_k^2 < \infty$ and
\begin{equation*}
\lim_{n \to \infty} \sum_{k=1}^\infty \left( (1+ \lambda^b_k)^{-\frac{\alpha}{2}} \Ebb \left[(X^{b,k}_t)^2 \right]^\frac{1}{2} \langle \phi^b_k, f^x_n\rangle_\mu - y_k \right)^2 = 0.
\end{equation*}
Since $\phi^b_k$ is continuous we have $\lim_{n \to \infty}\langle \phi^b_k, f^x_n\rangle_\mu = \phi^b_k(x)$. Thus by Fatou's lemma we can identify the sequence $(y_k)$. We must have
\begin{equation}\label{pointfin}
\sum_{k=1}^\infty (1+ \lambda^b_k)^{-\alpha} \Ebb \left[(X^{b,k}_t)^2 \right] \phi^b_k(x)^2 < \infty
\end{equation}
and
\begin{equation*}
\lim_{n \to \infty}\sum_{k=1}^\infty (1+ \lambda^b_k)^{-\alpha} \Ebb \left[(X^{b,k}_t)^2 \right] \left( \langle \phi^b_k, f^x_n\rangle_\mu - \phi^b_k(x) \right)^2 = 0.
\end{equation*}
Equivalently by \eqref{Wseries},
\begin{equation*}
\Ebb \left[ u(t,x)^2 \right] < \infty
\end{equation*}
(so we have proven square-integrability) and
\begin{equation*}
\lim_{n \to \infty}\Ebb \left[ \left( \langle u(t), f^x_n \rangle_\mu - u(t,x) \right)^2 \right] = 0.
\end{equation*}
In particular by taking $m \to \infty$ in \eqref{cauchyseq} we have that
\begin{equation*}
\Ebb \left[ \left( \langle u(t), f^x_n \rangle_\mu - u(t,x) \right)^2 \right] \leq 2c_6e^{2t} 2^{-n}.
\end{equation*}
\end{proof}
By virtue of the previous theorem it is possible to interpret solutions  $u$ to the SPDE \eqref{SPDE} as random maps $u: [0,\infty) \times F \to \Rbb$, where as usual we have suppressed the dependence of $u$ on the underlying probability space. It therefore makes sense to consider issues of continuity of $u$ on $[0,\infty) \times F$.
\begin{rem}
We note that \cite[Example 7.4]{foondun2011} gives a similar result for stochastic heat equations where the operator $L$ is the generator for a fractional diffusion in the sense of \cite[Section 3]{barlow1998} under suitable conditions.
\end{rem}

\section{H\"{o}lder regularity}\label{Holder}
The aim of this section is to use our continuity theorems of Section \ref{Kolmogorov} to prove H\"{o}lder regularity results for a version of the family defined in Theorem \ref{ptreg}, and then show that this version can be identified with the original solution to \eqref{SPDE}. We wish to use Theorem \ref{cty3} and Corollary \ref{cty4}, so we need estimates on the expected spatial and temporal increments of the solution.

\subsection{Spatial estimate}

\begin{prop}\label{spaceestim2}
Let $T > 0$. Let $u = ( u(t,x) )_{(t,x) \in [0,\infty) \times F}$ be the family defined in Theorem \ref{ptreg}. Then there exists a constant $C_2 > 0$ such that
\begin{equation*}
\Ebb \left[ (u(t,x) - u(t,y))^2 \right] \leq C_2R(x,y)
\end{equation*}
for all $t \in [0,T]$ and all $x,y \in F$.
\end{prop}
\begin{proof}
Recall from Theorem \ref{ptreg} that
\begin{equation*}
\lim_{n \to \infty}\Ebb \left[ \left( \left\langle u(t), f^x_n \right\rangle_\mu - u(t,x) \right)^2 \right] = 0,
\end{equation*}
and an analogous result holds for $y$. Thus by Lemma \ref{resolvestim},
\begin{equation*}
\begin{split}
\Ebb &\left[ (u(t,x) - u(t,y))^2 \right] = \lim_{n \to \infty} \Ebb \left[ \left\langle u(t), f^x_n - f^y_n \right\rangle_\mu^2 \right]\\
&\leq \frac{e^{2t}}{2} \lim_{n \to \infty}\int_F \int_F \rho^b_1(z_1,z_2)(f^x_n(z_1) - f^y_n(z_1))(f^x_n(z_2) - f^y_n(z_2))\mu(dz_1)\mu(dz_2)\\
&= \frac{e^{2t}}{2} \left( \rho^b_1(x,x) - 2\rho^b_1(x,y) + \rho^b_1(y,y) \right),
\end{split}
\end{equation*}
where we have used Proposition \ref{resolvLip}, Proposition \ref{nhoodestim} and the definition of $f^x_n$ (similarly to the proof of Theorem \ref{ptreg}). Hence again by Proposition \ref{resolvLip},
\begin{equation*}
\begin{split}
\Ebb \left[ (u(t,x) - u(t,y))^2 \right] &\leq \frac{e^{2T}}{2} \left( \rho^b_1(x,x) - \rho^b_1(x,y) + \rho^b_1(y,y) - \rho^b_1(y,x) \right)\\
&\leq 2 e^{2T} R(x,y).
\end{split}
\end{equation*}
\end{proof}

\subsection{Temporal estimates}

For the time estimates we can save ourselves some work by noticing that if $X^{b,k}$ is an Ornstein-Uhlenbeck process then
\begin{equation*}
\Ebb \left[ (X^{b,k}_s - X^{b,k}_{s+t})^2 \right] = \frac{1}{\lambda^b_k}(1-e^{-\lambda^b_kt}) - \frac{e^{-2\lambda^b_ks}}{2\lambda^b_k} (1- e^{-\lambda^b_kt})^2,
\end{equation*}
so that
\begin{equation*}
\frac{1}{2\lambda^b_k}(1-e^{-\lambda^b_kt}) \leq \Ebb \left[ (X^{b,k}_s - X^{b,k}_{s+t})^2 \right] \leq \frac{1}{\lambda^b_k}(1-e^{-\lambda^b_kt})
\end{equation*}
for any $s,t \in [0,\infty)$. Therefore regardless of whether $X^{b,k}$ is an Ornstein-Uhlenbeck or Wiener process we have that
\begin{equation*}
\Ebb \left[ (X^{b,k}_s - X^{b,k}_{s+t})^2 \right] \leq 2\Ebb \left[ (X^{b,k}_t)^2 \right].
\end{equation*}
Now since (using the independence of the $X^{b,k}$)
\begin{equation*}
\Ebb \left[ (u(s,x) - u(s+t,x))^2 \right] = \sum_{k=1}^\infty (1+ \lambda^b_k)^{-\alpha} \Ebb \left[(X^{b,k}_s - X^{b,k}_{s+t})^2 \right] \phi^b_k(x)^2,
\end{equation*}
it follows that it suffices to find estimates of the above in the case $s=0$.

We start with a method similar to the proof of \cite[Proposition 3.7]{walsh1986} which does not quite cover all values of $\alpha$. First, a lemma:
\begin{lem}\label{absum}
Consider the sum
\begin{equation*}
\sigma_{ab}(t) = \sum_{k=1}^\infty \left( k^{a-1} \wedge (k^{b-1}t) \right)
\end{equation*}
for $t \in (0,\infty)$, where $a,b \in \Rbb$ are constants. Then the following hold:
\begin{enumerate}
\item If $a, b \geq 0$ then $\sigma_{ab}(t)$ diverges.
\item If $a \in \Rbb$, $b < 0$ then there exists $C_{a,b} > 0$ such that $\sigma_{ab}(t) \leq C_{a,b}t$ for all $t$.
\item If $a < 0$, $b \geq 0$ then there exists $C_{a,b} > 0$ such that $\sigma_{ab}(t) \leq C_{a,b}t^{\frac{-a}{b - a}}$ for all $t$.
\end{enumerate}
\end{lem}
\begin{proof}
(1) is obvious. For (2) take $C_{a,b} = \zeta(1-b)$ the Riemann zeta function.

For (3) we must consider two cases depending on the value of $b$. First assume that $b \in [0,1]$. Then $x \mapsto (x^{a-1} \wedge (x^{b-1}t))$ is a decreasing function on $(0,\infty)$ so
\begin{equation*}
\begin{split}
\sigma_{ab}(t) &\leq \int_0^\infty (x^{a-1} \wedge (x^{b-1}t))dx\\
&= t\int_0^{t^{\frac{-1}{b-a}}}x^{b-1}dx + \int_{t^{\frac{-1}{b-a}}}^\infty x^{a-1}dx\\
&= tb^{-1}t^{\frac{-b}{b-a}} - a^{-1}t^{\frac{-a}{b-a}}\\
&= C_{a,b}t^{\frac{-a}{b-a}}
\end{split}
\end{equation*}
where we have $C_{a,b} = b^{-1} - a^{-1}$. If $b > 1$ then $x \mapsto (x^{a-1} \wedge (x^{b-1}t))$ is increasing on $[0,t^{\frac{-1}{b-a}}]$ where it is equal to $x^{b-1}t$ and decreasing on $[t^{\frac{-1}{b-a}},\infty)$ where it is equal to $x^{a-1}$. Thus
\begin{equation*}
k^{a-1} \wedge (k^{b-1}t) \leq k^{a-1} \wedge (t^{-\frac{b-1}{b-a}}t) = k^{a-1} \wedge (k^{1-1}t^{\frac{1-a}{b-a}})
\end{equation*}
for all $k \in \Nbb$, and we are back to the case $a < 0$, $b = 1$. It follows that
\begin{equation*}
\sigma_{ab}(t) \leq (1 - a^{-1}) \left(t^{\frac{1-a}{b-a}} \right)^{\frac{-a}{1-a}} = (1 - a^{-1})t^{\frac{-a}{b-a}}
\end{equation*}
so we have $C_{a,b} = 1 - a^{-1}$.
\end{proof}
\begin{prop}\label{timeestim1}
Let $T > 0$. Let $u = ( u(t,x) )_{(t,x) \in [0,\infty) \times F}$ be the family defined in Theorem \ref{ptreg}. If $\alpha > d_s - 1$ then there exists a constant $C_3 > 0$ such that
\begin{equation*}
\Ebb \left[ (u(s,x) - u(t,x))^2 \right] \leq C_3|s-t|^{1 \wedge (1 - d_s + \alpha)}
\end{equation*}
for all $s,t \in [0,T]$ and all $x \in F$.
\end{prop}
\begin{proof}
We assume that $\lambda^b_1 > 0$ to streamline our calculations. The case $\lambda^b_1 = 0$ is left as an exercise. By the discussion at the start of this section we may assume that $s=0$. Fix $x \in F$ and $t \in [0,T]$. Recall the constant $c_3$ from Proposition \ref{spectra}. By independence of the $X^{b,k}$ we have that
\begin{equation*}
\begin{split}
\Ebb \left[ u(t,x)^2 \right] &= \Ebb \left[ \left( \sum_{k=1}^\infty (1+ \lambda^b_k)^{-\frac{\alpha}{2}} X^{b,k}_t \phi^b_k(x) \right)^2 \right]\\
&\leq \sum_{k=1}^\infty \frac{\phi^b_k(x)^2}{(1+ \lambda^b_k)^{\alpha}} \frac{1}{\lambda^b_k}(1-e^{-\lambda^b_kt})\\
&\leq c_3^2\sum_{k=1}^\infty \frac{1-e^{-\lambda^b_kt}}{(\lambda^b_k)^{1 - \frac{d_s}{2}}(1+ \lambda^b_k)^{\alpha}}\\
&\leq c_3^2\sum_{k=1}^\infty \frac{1 \wedge (\lambda^b_kt)}{(\lambda^b_k)^{1 - \frac{d_s}{2}}(1+ \lambda^b_k)^{\alpha}}\\
&\leq c'\sum_{k=1}^\infty \left( k^{1- \frac{2+2\alpha}{d_s}} \wedge (k^{1- \frac{2\alpha}{d_s}}t) \right)
\end{split}
\end{equation*}
and we are within the scope of Lemma \ref{absum} (as long as $t > 0$, though the case $t=0$ is trivial). We find that if $\alpha > d_s-1$ then the sum converges and there exists $c'' > 0$ such that
\begin{equation*}
\Ebb \left[ u(t,x)^2 \right] \leq c''t^{1 \wedge (1 + \alpha - d_s)}
\end{equation*}
for all $x \in F$, $t \in [0,T]$.
\end{proof}
\begin{rem}\label{intervaltime}
In the case of Example \ref{interval} with $F = [0,1]$ the above estimate can be made to work for all $\alpha \geq 0$. This is because in this case $d_s = 1$ and the eigenfunctions $\phi^b_k$ satisfy $\Vert \phi^b_k \Vert_{\infty} \leq c$ for some $c > 0$ so we instead have that
\begin{equation*}
\Ebb \left[ u(t,x)^2 \right] \leq c'''\sum_{k=1}^\infty \left( k^{-(2+2\alpha)} \wedge (k^{-2\alpha}t) \right).
\end{equation*}
Therefore using Lemma \ref{absum}, we have for all $\alpha \geq 0$ that
\begin{equation*}
\Ebb \left[ (u(s,x) - u(t,x))^2 \right] \leq C_3|s-t|^{1 \wedge (\frac{1}{2} + \alpha)}.
\end{equation*}
This is the method used in \cite{walsh1986}.
\end{rem}
We now prove an alternative estimate that is weaker for large $\alpha$ but holds for all $\alpha \geq 0$.
\begin{prop}\label{timeestim2}
Let $T > 0$. Let $u = ( u(t,x) )_{(t,x) \in [0,\infty) \times F}$ be the family defined in Theorem \ref{ptreg}. Then there exists $C_4 > 0$ such that
\begin{equation*}
\Ebb \left[ (u(s,x) - u(t,x))^2 \right] \leq C_4|s-t|^{1 - \frac{d_s}{2}}
\end{equation*}
for all $s,t \in [0,T]$ and all $x \in F$.
\end{prop}
\begin{proof}
By the discussion at the start of this section we may assume that $s=0$. Set
\begin{equation*}
c_8' = (c_8 e^{2T}) \vee \frac{Td_H}{r_{\min}^{d_H}}
\end{equation*}
where the constant $c_8$ is from Theorem \ref{ptreg}. By Theorem \ref{ptreg} and It\={o}'s isometry (see proof of Lemma \ref{resolvestim}) we have that if $n \geq 0$ is an integer then
\begin{equation*}
\begin{split}
\Ebb \left[ u(t,x)^2 \right] &\leq 2\Ebb \left[ \langle u(t), f^x_n \rangle_\mu^2 \right] + 2c_8e^{2t} 2^{-n}\\
&= 2\int_0^t \Vert (1-L_b)^{-\frac{\alpha}{2}}S^b_s f^x_n \Vert_\mu^2 ds + 2c_8e^{2t} 2^{-n}\\
&\leq 2\int_0^t \Vert (1-L_b)^{-\frac{\alpha}{2}}S^b_s f^x_n \Vert_\mu^2 ds + 2c_8' 2^{-n}.
\end{split}
\end{equation*}
By the functional calculus for self-adjoint operators, $\Vert (1-L_b)^{-\frac{\alpha}{2}}S^b_s \Vert \leq 1$ for all $s \geq 0$. Thus
\begin{equation*}
\begin{split}
\Ebb \left[ u(t,x)^2 \right] &\leq 2t \Vert f^x_n \Vert_\mu^2 + 2c_8' 2^{-n}\\
&\leq 2r_{\min}^{-d_H}t 2^{d_Hn} + 2c_8' 2^{-n}
\end{split}
\end{equation*}
for all $(t,x) \in [0,T] \times F$ and all integer $n \geq 0$. We assume now that $t > 0$, and our aim is to choose $n \geq 0$ to minimise the above expression. Fixing $t \in (0,T]$, define $g: \Rbb \to [0,\infty)$ such that $g(y) = r_{\min}^{-d_H}t 2^{d_Hy} + c_8'2^{-y}$. The function $g$ has a unique stationary point which is a global minimum at
\begin{equation*}
y_0 = \frac{1}{(d_H + 1)\log 2} \log \left( \frac{r_{\min}^{d_H}c_8'}{d_Ht} \right).
\end{equation*}
Since $t \leq T$ we have by the definition of $c_8'$ that $y_0 \geq 0$. Since $y_0$ is not necessarily an integer we choose $n = \lceil y_0 \rceil$. Then $g$ is increasing in $[y_0,\infty)$ so we have that
\begin{equation*}
\Ebb \left[ u(t,x)^2 \right] \leq 2g(n) \leq 2g(y_0 + 1).
\end{equation*}
Setting $c_8'' := c_8'\frac{r_{\min}^{d_H}}{d_H}$ and evaluating the right-hand side we see that
\begin{equation*}
\begin{split}
\Ebb \left[ u(t,x)^2 \right] &\leq 2t r_{\min}^{-d_H} 2^{d_H} \left( \frac{c''_9}{t} \right)^\frac{d_H}{d_H + 1} + 2c_8' 2^{-1} \left( \frac{c''_9}{t} \right)^\frac{-1}{d_H + 1}\\
&\leq c_8''' t^\frac{1}{d_H + 1}\\
&= c_8''' t^{1 - \frac{d_s}{2}}
\end{split}
\end{equation*}
for all $(t,x) \in (0,T] \times F$, where the constant $c_8''' > 0$ is independent of $(t,x)$. This inequality obviously also holds in the case $t=0$.
\end{proof}

\subsection{H\"{o}lder regularity of solution}
We are now ready to prove the H\"{o}lder regularity result. It will turn out that the continuous version of $u(t,x)$ can be interpreted as an $\Hcal$-valued process, and is a version of the original $\Hcal$-valued solution to \eqref{SPDE} found in Theorem \ref{existence}. Recall $R_\infty$ the natural supremum metric on $\Rbb \times F$ given by
\begin{equation*}
R_\infty((s,x),(t,y)) = \max\{ |s-t|, R(x,y) \}.
\end{equation*}
\begin{thm}[H\"{o}lder regularity]\label{SPDEreg2}
Let $u = ( u(t,x) )_{(t,x) \in [0,\infty) \times F}$ be the family defined in Theorem \ref{ptreg}. Let
\begin{equation*}
\delta^\alpha =
\begin{cases}
\begin{array}{ll}
\frac{1}{2}(1 - \frac{d_s}{2}), & \alpha \leq \frac{d_s}{2},\\
\frac{1}{2}(1 - d_s + \alpha), & \frac{d_s}{2} < \alpha \leq d_s,\\
\frac{1}{2}, & \alpha > d_s.
\end{array}
\end{cases}
\end{equation*}

Then there exists a version $\tilde{u} = (\tilde{u}(t,x))_{(t,x) \in [0,\infty) \times F}$ of $u$ which satisfies the following:
\begin{enumerate}
\item For each $T > 0$, $\tilde{u}$ is almost surely essentially $\delta^\alpha$-H\"{o}lder continuous on $[0,T] \times F$ with respect to $R_\infty$.
\item For each $t \in [0,\infty)$, $\tilde{u}(t,\cdot)$ is almost surely essentially $\frac{1}{2}$-H\"{o}lder continuous on $F$ with respect to $R$.
\end{enumerate}
\end{thm}
\begin{proof}
Take $T > 0$ and consider $u_T$, the restriction of $u$ to $[0,T] \times F$. It is an easily verifiable fact that for every $p \in \Nbb$ there exists a constant $C_p' > 0$ such that if $Z$ is any centred real Gaussian random variable then
\begin{equation*}
\Ebb [Z^{2p}] = C_p' \Ebb[Z^2]^p.
\end{equation*}
We also know that $u_T$ is a centred Gaussian process on $[0,T] \times F$ by Theorem \ref{ptreg}.

We will treat the case $\alpha \leq \frac{d_s}{2}$, which is precisely the region of values of $\alpha$ for which Proposition \ref{timeestim2} will give us a better temporal H\"{o}lder exponent than Proposition \ref{timeestim1}. Propositions \ref{spaceestim2} and \ref{timeestim2} then give us the estimates
\begin{equation}\label{pestims}
\begin{split}
\Ebb \left[ (u_T(t,x) - u_T(t,y))^{2p} \right] &\leq C_p'C_2^pR(x,y)^p,\\
\Ebb \left[ (u_T(s,x) - u_T(t,x))^{2p} \right] &\leq C_p'C_4^p|s-t|^{p(1 - \frac{d_s}{2})}
\end{split}
\end{equation}
for all $s,t \in [0,T]$ and all $x,y \in F$. Taking $p$ arbitrarily large and then using Corollary \ref{cty4} we get a version $\tilde{u}_T$ of $u_T$ (that is, $\tilde{u}_T$ is a version of $u$ on $[0,T] \times F$) that satisfies the H\"{o}lder regularity conditions of the theorem for the given value of $T$. This works because any two almost surely continuous versions of $u_T$ must coincide almost surely since $[0,T] \times F$ is separable.

If now $T' > T$ and we construct a version $\tilde{u}_{T'}$ of $u$ on $[0,T'] \times F$ in the same way, then $\tilde{u}_{T'}$ must agree with $\tilde{u}_T$ on $[0,T] \times F$ almost surely since both are almost surely continuous on $[0,T] \times F$ which is separable. Therefore let $T = n$ for $n \in \Nbb$ and let $\Omega'$ be the almost sure event
\begin{equation*}
\Omega' = \bigcap_{n=1}^\infty \left\{ \text{$\tilde{u}_{n+1}$ agrees with $\tilde{u}_n$ on $[0,n] \times F$} \right\}.
\end{equation*}
Then for $(t,x) \in [0,\infty) \times F$ define $\tilde{u}(t,x) = \tilde{u}_{\lceil t \rceil + 1}(t,x)$ on $\Omega'$ and $\tilde{u}(t,x) = 0$ otherwise, and we are done.

Now if $\alpha > \frac{d_s}{2}$ then we use the temporal estimate of Proposition \ref{timeestim1} rather than the temporal estimate of Proposition \ref{timeestim2} in \eqref{pestims}.
\end{proof}
\begin{thm}[Continuous version is version of original solution]\label{thm:ctsver}
The collection of random variables $\tilde{u} = (\tilde{u}(t,x))_{(t,x) \in [0,\infty) \times F}$ constructed in Theorem \ref{SPDEreg2} is such that $(\tilde{u}(t,\cdot))_{t \in [0,\infty)}$ is an $\Hcal$-valued process, and moreover $(\tilde{u}(t,\cdot))_{t \in [0,\infty)}$ is an $\Hcal$-continuous version of the $\Hcal$-valued solution to \eqref{SPDE} found in Theorem \ref{existence} (with $u_0 = 0$).
\end{thm}
\begin{proof}
From Theorem \ref{SPDEreg2}, $\tilde{u}$ is almost surely continuous in $[0,\infty) \times F$. Each $\tilde{u}(t,x)$ is a well-defined random variable so by \cite[Lemma 4.51]{aliprantis2006}, $\tilde{u}$ is jointly measurable. Using continuity again, this implies that $\tilde{u}(t,\cdot) \in \Hcal$ for all $t \in [0,\infty)$ almost surely. We also have that $t \mapsto \tilde{u}(t,\cdot)$ is a continuous function from $[0,\infty)$ to $\Hcal$ for each $\omega \in \Omega$ and that each $\tilde{u}(t,\cdot)$ is a Borel measurable map from $\Omega$ to $\Hcal$; the latter follows from the joint continuity of $\tilde{u}$ and the fact that the Borel $\sigma$-algebra of $\Hcal$ is generated by the bounded linear functionals on $\Hcal$.

For each $n \geq 1$ and $(t,x) \in [0,\infty) \times F$, define
\begin{equation*}
u^{(n)}(t,x) = \sum_{k=1}^n (1+ \lambda^b_k)^{-\frac{\alpha}{2}} X^{b,k}_t \phi^b_k(x).
\end{equation*}
Then each $u^{(n)}$ is obviously jointly measurable, and by \eqref{pointfin}, $u^{(n)}(t,x) \to \tilde{u}(t,x)$ in $\Lcal^2(\Pbb)$ for each $(t,x) \in [0,\infty) \times F$. In fact by joint measurability we also have that
\begin{equation*}
\begin{split}
\Ebb\left[ \int_F \left( u^{(n)}(t,x) - \tilde{u}(t,x) \right)^2 \mu(dx) \right] &= \int_F \Ebb\left[ \left( u^{(n)}(t,x) - \tilde{u}(t,x) \right)^2 \right] \mu(dx)\\
&= \int_F \Ebb\left[ \left( \sum_{k=n+1}^\infty (1+ \lambda^b_k)^{-\frac{\alpha}{2}} X^{b,k}_t \phi^b_k(x) \right)^2 \right] \mu(dx)\\
&= \int_F \sum_{k=n+1}^\infty (1+ \lambda^b_k)^{-\alpha} \Ebb\left[ \left( X^{b,k}_t \right)^2 \right] \phi^b_k(x)^2 \mu(dx)\\
\end{split}
\end{equation*}
where in the last line we have used \eqref{pointfin}. Then by Tonelli's theorem,
\begin{equation*}
\begin{split}
\Ebb\left[ \int_F \left( u^{(n)}(t,x) - \tilde{u}(t,x) \right)^2 \mu(dx) \right] &= \sum_{k=n+1}^\infty (1+ \lambda^b_k)^{-\alpha} \Ebb\left[ \left( X^{b,k}_t \right)^2 \right]\\
&\to 0
\end{split}
\end{equation*}
as $n \to \infty$, by the fact that $\sum_{k=1}^\infty \Ebb[ ( X^{b,k}_t )^2 ] < \infty$. In particular, this implies that $\langle u^{(n)}(t,\cdot) , \phi^b_k \rangle_\mu \to \langle \tilde{u}(t,\cdot) , \phi^b_k \rangle_\mu$ in $\Lcal^2(\Pbb)$ as $n \to \infty$, for every $t \in [0,\infty)$ and every $k \geq 1$.

Recall that if $u_0 = 0$ then the solution to \eqref{SPDE} is simply given by the series
\begin{equation*}
W^b_\alpha(t) = \sum_{k=1}^\infty (1+ \lambda^b_k)^{-\frac{\alpha}{2}} X^{b,k}_t \phi^b_k.
\end{equation*}
It follows that for all $t \in [0,\infty)$ and $n \geq k$ we have that
\begin{equation*}
\langle u^{(n)}(t,\cdot) , \phi^b_k \rangle_\mu = (1+ \lambda^b_k)^{-\frac{\alpha}{2}} X^{b,k}_t = \langle W^b_\alpha(t) , \phi^b_k \rangle_\mu
\end{equation*}
almost surely. Therefore $\tilde{u}(t,\cdot) = W^b_\alpha(t)$ almost surely for all $t \in [0,\infty)$ and we are done.
\end{proof}

\begin{rem}\label{euclidequiv}
In \cite{hu2006} it is shown that under some mild conditions on the p.c.f.s.s. set $(F,(\psi_i)_{i=1}^M)$, it can be embedded into Euclidean space in such a way that its resistance metric $R$ is uniformly equivalent to some power of the Euclidean metric. Therefore in this case the conclusion of Theorem \ref{SPDEreg2} holds with respect to the spatial Euclidean metric, albeit with a different H\"{o}lder exponent. An example given in \cite[Section 3]{hu2006} is the $n$-dimensional Sierpinski gasket for $n \geq 2$, see Example \ref{examples}(2) of the present paper. This fractal has a natural embedding in $\Rbb^n$, and it is shown that in this case we have a constant $c > 0$ such that
\begin{equation*}
c^{-1}|x-y|^{d_w - d_f} \leq R(x,y) \leq c|x-y|^{d_w - d_f}
\end{equation*}
for all $x,y \in F \subseteq \Rbb^n$, where $d_w = \frac{\log(n+3)}{\log2}$ is the \textit{walk dimension} of the gasket and $d_f = \frac{\log(n+1)}{\log2}$ is its Euclidean Hausdorff dimension. These fractals all admit function-valued solutions to their respective SHEs but their ambient spaces $\Rbb^n$ do not.
\end{rem}
\begin{rem}
From Theorem \ref{SPDEreg2} we see that the operator $(1 - L_b)^{-\frac{\alpha}{2}}$ has a smoothing effect on the solution to \eqref{SPDE} as $\alpha$ increases. However the theorem suggests that this does not change the H\"{older} exponents of the solution until $\alpha$ reaches the value $\frac{d_s}{2}$. On the other hand, recall from Remark \ref{intervaltime} that if $F = [0,1]$ as in Example \ref{interval} then the temporal H\"{o}lder exponent of the solution to \eqref{SPDE}, viewed as a function of $\alpha$, is linearly strictly increasing in some neighbourhood of $\alpha = 0$. Intuitively this phenomenon should occur for any $F$. We therefore conjecture that the H\"{older} exponent in Theorem \ref{SPDEreg2}(1) is not sharp when $\alpha \in (0,d_s)$, and is in fact equal to the exponent obtained in Theorem \ref{existence} for the solution interpreted as an $\Hcal$-valued process.
\end{rem}
We have shown regularity properties of the solution to \eqref{SPDE} in the case $u_0 = 0$, and henceforth we assume that we are dealing with the continuous version of this solution. If we now take an arbitrary initial condition $u_0 = f \in \Hcal$, then obviously the same results may not hold since $f$ may be very rough. We can however prove continuity in \textit{almost} the entire domain $[0,\infty) \times F$.
\begin{thm}
Let $u:[0,\infty) \times F \to \Rbb$ be the solution to \eqref{SPDE} with initial condition $u_0 = f \in \Hcal$. Then $u$ (has a version which) is almost surely continuous in $(0,\infty) \times F$ with respect to $R_\infty$. Moreover, if either
\begin{enumerate}
\item $b=N$ and $f$ is continuous on $F$, or
\item $b=D$ and $f$ is continuous on $F$ with $f|_{F^0} \equiv 0$,
\end{enumerate}
then $u$ (has a version which) is almost surely continuous in $[0,\infty) \times F$ with respect to $R_\infty$.
\end{thm}
\begin{proof}
Theorem \ref{SPDEreg2} gives us that the map $(t,x) \mapsto W^b_\alpha(t)(x)$ (has a version which) is almost surely continuous in $[0,\infty) \times F$. We know that if $t > 0$ then $S^b_t$ maps $\Hcal$ into $\Dcal$, and in particular into the space of continuous functions. Thus it makes sense to talk about $S^b_tf(x)$ for $x \in F$. Define
\begin{equation*}
u(t,x) := S^b_tf(x) + W^b_\alpha(t)(x),
\end{equation*}
then it suffices to prove continuity of the map $(t,x) \mapsto S^b_tf(x)$. This follows in the same way as the proof of \cite[Proposition 5.2.6]{kigami2001}. The last two statements in the theorem are immediate corollaries of \cite[Proposition 5.2.6]{kigami2001}.
\end{proof}

\section{Invariant measure}\label{invar}
We conclude with a brief description of the long-time behaviour of the solutions to \eqref{SPDE}. In this section we allow the initial condition $u_0$ to be an $\Hcal$-valued random variable which is independent of $W$.
\begin{defn}
An \textit{invariant measure} for the SPDE \eqref{SPDE} is a probability measure $\nu_\infty$ on $\Lcal^2(F,\mu) = \Hcal$ such that if $u$ is the solution to \eqref{SPDE} with random initial condition $u_0 \sim \nu_\infty$ (independent of $W$) then $u(t) \sim \nu_\infty$ for all $t > 0$.
\end{defn}
In the following theorems, let $(Z_k)_{k=1}^\infty$ be a sequence of independent and identically distributed one-dimensional standard Gaussian random variables.
\begin{thm}
Let $b=D$. Then \eqref{SPDE} has a unique invariant measure $\nu_\infty^D$ and it is given by
\begin{equation*}
\nu_\infty^D = \Law \left( \sum_{k=1}^\infty (1 + \lambda^D_k)^{-\frac{\alpha}{2}} (2\lambda^D_k)^{-\frac{1}{2}} Z_k \phi^D_k \right).
\end{equation*}
If $u$ is a solution to \eqref{SPDE} in the case $b=D$ then $u(t)$ converges weakly to $\nu_\infty^D$ as $t \to \infty$ regardless of its initial distribution $\Law(u_0)$.
\end{thm}
\begin{proof}
We first show that the definition of $\nu_\infty^D$ makes sense. We have that
\begin{equation*}
\begin{split}
\Ebb\left[ \sum_{k=1}^\infty \left\Vert (1 + \lambda^D_k)^{-\frac{\alpha}{2}} (2\lambda^D_k)^{-\frac{1}{2}} Z_k \phi^D_k \right\Vert_\mu^2 \right] &= \Ebb\left[ \sum_{k=1}^\infty (1 + \lambda^D_k)^{-\alpha} (2\lambda^D_k)^{-1} Z_k^2 \right]\\
&\leq \sum_{k=1}^\infty (1 + \lambda^D_k)^{-\alpha} (2\lambda^D_k)^{-1}\\
&\leq \frac{c_1^{-1-\alpha}}{2} \sum_{k=1}^\infty k^{-\frac{2}{d_s}(1+\alpha)}\\
&< \infty,
\end{split}
\end{equation*}
so $\nu_\infty^D$ is indeed a well-defined probability measure on $\Hcal$. Now suppose $u$ has initial distribution $u_0 \sim \nu_\infty^D$. By Definition \ref{mild} and \eqref{mild0} we have that
\begin{equation*}
u(t) = \sum_{k=1}^\infty (1 + \lambda^D_k)^{-\frac{\alpha}{2}} \left( e^{-\lambda^D_kt} (2\lambda^D_k)^{-\frac{1}{2}} Z_k + X^{D,k}_t \right) \phi^D_k,
\end{equation*}
where $(Z_k)_{k=1}^\infty$ and $(X^{D,k})_{k=1}^\infty$ are understood to be independent. Recall that $\lambda^D_1 > 0$ by Remark \ref{consteval}, so for every $k$, $X^{D,k}$ is a centred Ornstein-Uhlenbeck process with unit volatility and rate parameter $\lambda^D_k$. Moreover, $t \mapsto e^{-\lambda^D_kt} (2\lambda^D_k)^{-\frac{1}{2}} Z_k + X^{D,k}_t$ is an Ornstein-Uhlenbeck process with unit volatility, rate parameter $\lambda^D_k$ and initial distribution given by the law of $(2\lambda^D_k)^{-\frac{1}{2}} Z_k$, which turns out to be exactly its invariant measure (which we leave as an exercise for the reader). Thus the law of $u(t)$ is equal to $\nu_\infty^D$ for all $t > 0$, so $\nu_\infty^D$ is an invariant measure. We also see that for all $f \in \Hcal$,
\begin{equation*}
\Vert S^D_tf \Vert_\mu^2 = \sum_{k=1}^\infty e^{-2\lambda^D_kt}f_k^2 \leq e^{-2\lambda^D_1t} \sum_{k=1}^\infty f_k^2 = e^{-2\lambda^D_1t} \Vert f \Vert_\mu^2
\end{equation*}
where $f_k = \langle \phi^D_k,f \rangle_\mu$, so $\lim_{t\to\infty} \Vert S^D_t \Vert = 0$. Then the uniqueness and weak convergence results are direct consequences of \cite[Theorem 11.20]{daprato2014} (or alternatively \cite[Proposition 5.23]{hairer2009}).
\end{proof}
In the case $b=N$ we do not have nearly as neat a result, but there exists a decomposition of $u$ into two independent processes, one of which has similar invariance properties to the $b=D$ case and the other of which is simply a Brownian motion.
\begin{defn}
Let $\Hcal_1 \subseteq \Hcal$ be the space spanned by $\phi^N_1$, which we recall from Remark \ref{consteval} to be the constant function $\phi^N_1 \equiv 1$. Let $\Hcal_1^\perp$ be its orthogonal complement. Let $\pi_1:\Hcal \to \Hcal_1$ be the orthogonal projection onto $\Hcal_1$ and let $\pi_1^\perp:\Hcal \to \Hcal_1^\perp$ be the orthogonal projection onto $\Hcal_1^\perp$.

Let $\star$ denote convolution of measures. For a measure $\nu$ on $\Hcal$, let $\pi_1^*\nu$ denote the pushforward of $\nu$ with respect to $\pi_1$, which is a measure on $\Hcal_1$.
\end{defn}
\begin{thm}
Define the probability measure $\nu_\infty^N$ on $\Hcal_1^\perp$ by
\begin{equation*}
\nu_\infty^N = \Law \left( \sum_{k=2}^\infty (1 + \lambda^N_k)^{-\frac{\alpha}{2}} (2\lambda^N_k)^{-\frac{1}{2}} Z_k \phi^N_k \right).
\end{equation*}
Let $u$ be a solution to \eqref{SPDE} with $b=N$. Then there exists a one-dimensional standard Wiener process $B = (B(t))_{t \geq 0}$ which is adapted to the filtration generated by $W$ such that $B$ and $u-B$ are independent processes, and if $u$ has initial distribution $u_0 \sim \nu \star \nu_\infty^N$ for some probability measure $\nu$ on $\Hcal_1$ then $u(t) - B(t) \sim \nu \star \nu_\infty^N$ for all $t > 0$. Moreover, if $u$ has initial distribution $u_0 \sim \nu_0$ for some probability measure $\nu_0$ on $\Hcal$ then $u(t) - B(t)$ converges weakly to $(\pi_1^*\nu_0) \star \nu_\infty^N$ as $t \to \infty$.
\end{thm}
\begin{proof}
Recall that $\lambda^N_1 = 0$ and that $\lambda^N_k > 0$ for $k \geq 2$. Just as in the $b=D$ case we can prove that $\nu_\infty^N$ is a well-defined probability measure on $\Hcal_1^\perp$, and indeed on $\Hcal$ as well. We define
\begin{equation*}
B(t) := X^{N,1}_t = \int_0^t \phi^{N*}_1dW(s).
\end{equation*}
From our discussion after \eqref{OUproc} we know that $X^{N,1}$ is a standard one-dimensional Wiener process, and by Remark \ref{consteval}, $\phi^N_1 \equiv 1$. From the representation \eqref{Wseries} of $W^b_\alpha$ as a sum of independent stochastic processes it is clear that $u - B$ is independent of $B$.

Now suppose $u$ has initial distribution $u_0 \sim \nu \star \nu_\infty^N$ for some probability measure $\nu$ on $\Hcal_1$. The space $\Hcal_1$ is one-dimensional so let $Z_0$ be a real-valued random variable such that $Z_0\phi^N_1$ has law $\nu$. By Definition \ref{mild} and \eqref{mild0} we can write
\begin{equation*}
u(t) = \left( Z_0 + X^{N,1}_t \right) \phi^N_1 + \sum_{k=2}^\infty (1 + \lambda^N_k)^{-\frac{\alpha}{2}} \left( e^{-\lambda^N_kt} (2\lambda^N_k)^{-\frac{1}{2}} Z_k + X^{N,k}_t \right) \phi^N_k,
\end{equation*}
where again $Z_0$, $(Z_k)_{k=1}^\infty$ and $(X^{N,k})_{k=1}^\infty$ are understood to be independent. Now we observe that $\pi_1^\perp = \1bb_{(0,\infty)}(-L_N)$ using the functional calculus on self-adjoint operators, and in particular $\pi_1^\perp$ commutes with functions of $L_N$, including $S^N_t$. Then define
\begin{equation*}
\pi^\perp_1u(t) = u(t) - \left( Z_0 + X^{N,1}_t \right) =: u_1(t),
\end{equation*}
where we have identified scalars $c \in \Rbb$ with their associated constant functions $c\phi^N_1 \in \Hcal$.
It is then easily verifiable that $\pi_1^\perp W$ is a cylindrical Wiener process on $\Hcal_1^\perp$ and that $u_1$ is the (mild) solution to the SPDE on $\Hcal_1^\perp$ given by
\begin{equation}\label{projSPDE}
\begin{split}
du(t) &= L_Nu(t)dt + (1-L_N)^{-\frac{\alpha}{2}}\pi_1^\perp dW(t),\\
u(0) &= u_0 \in \Hcal_1^\perp
\end{split}
\end{equation}
with $u_0 \sim \nu_\infty^N$. Note that all operators in \eqref{projSPDE} commute with $\pi_1^\perp$ and so can be identified with their restriction to $\Hcal_1^\perp$ for the purposes of the above SPDE. By definition,
\begin{equation*}
u_1(t) = \sum_{k=2}^\infty (1 + \lambda^N_k)^{-\frac{\alpha}{2}} \left( e^{-\lambda^N_kt} (2\lambda^N_k)^{-\frac{1}{2}} Z_k + X^{N,k}_t \right) \phi^N_k.
\end{equation*}
Now just as in the $b=D$ case, using \cite[Theorem 11.20]{daprato2014} (or \cite[Proposition 5.23]{hairer2009}) we find that $\nu_\infty^N$ is the unique invariant measure for \eqref{projSPDE} and that the solution to \eqref{projSPDE} converges weakly to $\nu_\infty^N$ for any initial distribution on $\Hcal_1^\perp$. So we have that for all $t > 0$,
\begin{equation*}
u(t) - B(t) = Z_0 + u_1(t) \sim \nu \star \nu_\infty^N
\end{equation*}
as required.

We observe that if $u$ has deterministic initial value $u_0 = f \in \Hcal$ then this is equivalent to $u_0$ being distributed according to the convolution of Dirac measures $\delta_{f_1} \star \delta_{f_2}$ for $f_1 = \pi_1(f) \in \Hcal_1$, $f_2 = \pi_1^\perp(f) \in \Hcal_1^\perp$. By doing the usual eigenfunction expansion we have that $u(t) = f_1 + X^{N,1} + u_1(t)$ where $u_1$ is now the solution to \eqref{projSPDE} with initial value $f_2$. Thus $u(t) - B(t) = f_1 + u_1(t)$ which converges weakly to $\delta_{f_1} \star \nu_\infty^N$. Now assume $u$ has an arbitrary initial probability distribution $u_0 \sim \nu_0$ in $\Hcal$. By conditioning first on the value of $\pi_1(u_0) \in \Hcal_1$, then on the value of $\pi_1^\perp(u_0) \in \Hcal_1^\perp$ and then using the dominated convergence theorem we find that $\Ebb [g(u(t) - B(t))]$ converges to $ \left( (\pi_1^* \nu_0 \right) \star \nu_\infty^N)(g)$ for any continuous and bounded function $g$ on $\Hcal$. So we have weak convergence of $u(t) - B(t)$ to $(\pi_1^* \nu_0) \star \nu_\infty^N$.
\end{proof}

\section*{Acknowledgement}
The authors would like to thank an anonymous referee for simplifying their original proof of Proposition~\ref{resolvLip}.
The first author's work was supported by the ERC funded project ESig.
The work of the second author was supported by a United Kingdom Engineering and Physical Sciences Research Council studentship.

\bibliographystyle{alpha}

\end{document}